\documentclass[11pt,a4paper,reqno]{amsart}
\usepackage{graphicx}
\usepackage[linktocpage=true,colorlinks,citecolor=blue,linkcolor=blue,urlcolor=black]{hyperref}
\usepackage{amssymb}
\usepackage{cite}
\usepackage{amsmath}
\usepackage{latexsym}
\usepackage{amscd}
\usepackage{tikz}
\usepackage{amsthm}
\usepackage{mathrsfs}
\usepackage{url}
\usepackage[utf8]{inputenc}
\usepackage[english]{babel}
\usepackage{amsfonts}
\usepackage[textwidth=20mm]{todonotes}

\numberwithin{equation}{section}
\def\Zp{\mathbb{Z}_p}
\def\Zn{\mathbb{Z}_n}
\def\R{\mathbb R}
\def\Z{\mathbb Z}
\def\C{\mathbb C}

\def\N{\mathbb N}

\def\CA{\mathcal A}

\def\gcd{\operatorname{gcd}}
\renewcommand\Re{\operatorname{Re}}

\DeclareMathOperator\disc{disc}
\DeclareMathOperator\AP{AP}
\DeclareMathOperator\sgn{sgn}
\DeclareMathOperator\herdisc{herdisc}
\newcommand\AND{\mathrm{and}}
\newcommand\Forall{\mathrm{\;for\;all\;}}
\newcommand{\Mod}[1]{\;(\mathrm{mod}\;#1)}

\newtheorem{theorem}{Theorem}[section]
\newtheorem{lemma}[theorem]{Lemma}
\newtheorem{proposition}[theorem]{Proposition}
\newtheorem*{theorem*}{Theorem}

\newtheorem{corollary}[theorem]{Corollary}

\newtheorem{problem}{Problem}[section]

\theoremstyle{remark}
\newtheorem{remark}[theorem]{Remark}
\newtheorem*{notations}{Notations}

\theoremstyle{definition}

\theoremstyle{remark}

\numberwithin{equation}{section}

\begin{document}
	\title{Discrepancy in Modular Arithmetic Progressions}
	\author{Jacob Fox \and Max Wenqiang Xu \and Yunkun Zhou}
	\thanks{Fox is supported by a Packard Fellowship and by NSF award DMS-1855635. \\
		\indent Xu is supported by the Cuthbert C. Hurd Graduate Fellowship in the Mathematical Sciences, Stanford. \\
		\indent Zhou is supported by NSF GRFP Grant DGE-1656518.}
	\address{Department of Mathematics, Stanford University, Stanford, CA, USA}
	\email{\{jacobfox,maxxu,yunkunzhou\}@stanford.edu}

	\maketitle
	\begin{abstract}
		Celebrated theorems of Roth and of Matou\v{s}ek and Spencer together show that the discrepancy of arithmetic progressions in the first $n$ positive integers is $\Theta(n^{1/4})$. We study the analogous problem in the $\mathbb{Z}_n$ setting. We asymptotically determine the logarithm of the discrepancy of arithmetic progressions in $\Zn$ for all positive integer $n$. We further determine up to a constant factor the discrepancy of arithmetic progressions in $\Zn$ for many $n$. For example, if $n=p^k$ is a prime power, then 
		the discrepancy of arithmetic progressions in $\Zn$ is $\Theta(n^{1/3+r_k/(6k)})$, where $r_k \in \{0,1,2\}$ is the remainder when $k$ is divided by $3$. This solves a problem of Hebbinghaus and Srivastav.       
	\end{abstract}
	\section{Introduction}
Given a finite set $\Omega$, a (two-)coloring of $\Omega$ is a map $\chi: \Omega\to \{1, -1\}$, and a partial coloring is a map $\chi: \Omega \to \{-1, 0, 1\}$. For $A\subseteq \Omega$, let $\chi(A) = \sum_{x\in A}\chi(x).$ For a family $\mathcal{A}$ of subsets of $\Omega$, the discrepancy of $\mathcal{A}$ is defined to be  
\[\disc(\mathcal{A}): = \min_{\chi} \max_{A \in \mathcal{A}} |\chi(A)|,\]
where the minimum is over all colorings of $\Omega$. The discrepancy measures the guaranteed irregularity of colorings with respect to a set system. Discrepancy theory is a rich area of study, see the books \cite{AS,BC,Chazelle,Mat}. 

	In particular, the study of discrepancy of arithmetic progressions has a long history, including notable results of Weyl from 1916 \cite{Weyl} and Roth from 1964 \cite{Roth}. Let $[n]:= \{1,2,3,\dots,n\}$ and $\CA$ be the set of arithmetic progressions in $[n]$. Using Fourier analysis,
	Roth  \cite{Roth} proved that there is an absolute constant $c > 0$ such that
	\[\disc(\mathcal{A}) \geq cn^{\frac 14}.  \]
	The exponent $1/4$ was unexpected, as random colorings suggest that the best exponent might be $1/2$. Later, improving on a result of Montgomery and S\'ark\"ozy (see Problem 10 in \cite{Erdos-Sarkozy}), Beck \cite{Beck} proved that Roth's lower bound is sharp up to a polylogarithmic factor. It was a big challenge to remove the polylogarithmic factor and show that Roth's bound is sharp up to a constant factor, and it was finally done by Matou\v{s}ek and Spencer \cite{MS} via entropy and partial coloring methods.	
	
	The modular variant is also a very natural problem to study. For a positive integer $n$, an arithmetic progression in $\Zn$ is a set of the form $\{a + kd: 0\leq k < l\}$ for any $a, d\in \Zn$ and $l\geq 0$. To avoid repeated elements in the set, we may assume that $l \leq \frac{n}{\gcd(n, d)}$. Let $\mathcal{A}_n$ be the set of all arithmetic progressions in $\Zn$. The quantity we are interested in is 
	\[\disc(\mathcal{A}_n): = \min_{\chi: \Zn\to \{1,-1\} } \max_{A \in \mathcal{A}_n} |\chi(A)|.  \]
	In the case where $n=p$ is a prime, the following lower and upper bounds are proved by Hebbinghaus and Srivastav \cite{HS} and Alon and Spencer (Theorem 13.1.1 in \cite{AS}) respectively. There exist positive constants $c_1, c_2$ such that
	\[c_1  \sqrt{p} \le   \disc(\mathcal{A}_p) \le c_2  \sqrt{p\log p}.\]
	Hebbinghaus and Srivastav \cite{HS} wrote that it seemed to be a difficult open problem to close the $O(\sqrt{\log p})$ multiplicative gap between the upper and lower bounds. As part of our results, we remove the $\sqrt{\log p}$ factor in the upper bound and resolve the problem of determining $\disc(\CA_p)$ up to a constant factor. 
	
The problem for general $\Zn$ is more challenging and interesting. Note that any arithmetic progression in $[n]$ is also an arithmetic progression in $\Zn$, so Roth's lower bound from the integer case also applies. As the number of sets in $\CA_n$ is polynomial in $n$, considering a random coloring gives the following upper bound (see Theorem 13.1.1 in \cite{AS}). So there are $c_1,c_2 >0$ such that
	\[c_1 n^{\frac{1}{4}} \le \disc(\mathcal{A}_n) \le   c_2  n^{\frac{1}{2}} (\log n)^{\frac{1}{2}} . \] 
    Our main theorem asymptotically determines the logarithm of $\disc(\CA_n)$. It shows that $\disc(\CA_n)$ depends heavily on the 
    arithmetic structure of $n$, and neither of the above bounds are sharp. 
	
	Let $\omega(n)$ be the number of distinct prime factors of $n$ and $d(n)$ be the number of divisors of $n$.
	
	\begin{theorem}\label{main theorem}
		There exists an absolute constant $c > 0$ such that for any positive integer $n$,
		\[\frac{1}{8\sqrt{d(n)}} \cdot \min_{r|n}\left(\frac{n}{r} + \sqrt{r}\right) \le \disc(\CA_n) \leq \min_{r|n}\left(\frac{n}{r} + c \sqrt{r}\cdot 2^{\omega(r)}\right).\]
	\end{theorem}
	Note that we have $2^{\omega(r)}\leq 2^{\omega(n)} \leq d(n) = n^{o(1)}$ (see \cite{HW}).
	This implies that
	\[\disc(\CA_n) =  n^{\frac{1}{3} + x + o(1)},\]
	where $x \geq 0$ is the largest real number such that there is no factor of $n$ in the range $(n^{\frac{2}{3}-x}, n^{\frac{2}{3} +2x})$.
	Thus, our results determine the correct exponent of $n$ for $\disc(\CA_n)$ up to $o(1)$. 
	
	Notice that when $n$ has a bounded number of factors, Theorem \ref{main theorem} determines $\disc(\CA_n)$  up to a constant factor. In particular, when $n$ is prime, we have $\disc(\CA_n) = \Theta(n^{1/2})$, and this solves the problem of Hebbinghaus and Srivastav \cite{HS} discussed earlier. 
	
	We actually prove slightly stronger bounds, but we choose the formulation in Theorem~\ref{main theorem} for simplicity. These stronger bounds give the following sharp bound for $\disc(\CA_n)$ when $n$ is a prime power. 
	\begin{theorem}\label{prime power} If $n = p^k$ for prime $p$ and positive integer $k$, then
		\[  \disc(\CA_{n}) = \Theta\left( p^{\frac{k-\lfloor k/3\rfloor}{2}}\right) =  \Theta\left(n^{\frac{1}{3}+\frac{r_k}{6k}}\right),\]
		where $r_k \in \{0,1,2\}$ is the remainder when $k$ is divided by $3$. 
	\end{theorem}
	
	\begin{remark}\label{rem: unit 1}	
		It is also natural to study the discrepancy problem in the case where the coloring function $\chi$ is allowed to take any value in the unit circle in the complex plane $\{z\in \C: |z| = 1\}$.
		For example, this choice of coloring functions is studied in Tao's \cite{Tao} remarkable solution to the Erd\H{o}s discrepancy problem. By comparing the definitions of the coloring functions, we know that the discrepancy under the more general choice of $\chi$ is at most as large as the discrepancy under the choice $\chi: \Zn \to \{1,-1\}$. We remark that our proof is robust enough to give the same lower bound on discrepancy when $\chi$ is under the above choice. See Remark~\ref{rem:unit 2} for details on how to extend the proof of the lower bound to the more general setting.
		
	\end{remark}
	
	\begin{notations}Throughout the paper, all logarithms are base $e$ unless otherwise specified.
		
		We use symbols $c, c_0, c_1, c_2$, etc. to denote absolute constants. To avoid using too many symbols in different parts of the paper, we reuse these symbols to denote different constants. We make no attempt to optimize constant factors in our results.

		We treat elements in $\Zn$ as if they are in $\Z$ in the following ways. For an element $r\in \Zn$ and integers $a, b\in \Z$, we say $a\leq r \leq b$ if there exists an element $r'\in \Z$ such that $a\leq r' \leq b$ and $r' \equiv r\Mod n$. The notation is similar if either $\leq$ is replaced with $<$. We typically use it when $0\leq a \leq b < n$, so this notation should not cause any confusion. For any $a\in \Zn$, we may also define $\gcd(a, n) = \gcd(a', n)$ for any $a'\in \Z$ with $a' \equiv a\Mod n$. For two nonnegative integers $a$ and $b$, we write $a|b$ if $a$ divides $b$. For any factor $r$ of $n$ and $a\in \Zn$, we say that $r|a$ if $r|\gcd(a, n)$.
	\end{notations}
	
	\subsection*{Organization}
	In Section \ref{entropy}, we derive our first upper bound on $\disc(\CA_n)$ (see Corollary~\ref{cor:subset-coloring}). The upper bounds in Theorems~\ref{main theorem} and \ref{prime power} are proved in Section
	\ref{section upper} (see Theorem~\ref{thm:bound-from-distribution}). In Section~\ref{sec lower bounds}, we prove the lower bounds in Theorem~\ref{main theorem} (Corollary~\ref{cor:lower-bound-with-divisor}) and Theorem~\ref{prime power} (Corollary~\ref{cor:lower bound prime powers}). We end with some concluding remarks and open problems in Section \ref{concludingremarks}.

	\section{The First Step towards the Upper Bounds}\label{entropy}
	We use the following version of a lemma of Matou\v{s}ek and Spencer \cite{MS} to show there is a partial coloring that colors a constant fraction of the 
	elements of a set system with low discrepancy. The proof of this lemma uses the entropy method.
	\begin{lemma}[Section 4.6 in \cite{Mat}]\label{lem:partial-coloring-lemma}
		Let $(V, C)$ be a set system on $n$ elements, and let a number $\Delta_S \geq 2\sqrt{|S|}$ be given for each set $S\in C$. If
		\begin{equation}\sum_{S\in C:S\ne \emptyset} \exp\left(-\frac{\Delta_S^2}{4|S|}\right) \leq \frac{n}{50},
		\end{equation}
		then there is a partial coloring $\chi$ that assigns $\pm 1$ to at least $n/10$ variables (and $0$ to the rest) satisfying $|\chi(S)| \leq \Delta_S$ for each $S\in C$.
	\end{lemma}
	
	The following lemma shows that there is a partial coloring of a constant fraction of any subset $X$ of $\Zn$ such that modular arithmetic progressions have low discrepancy. The proof utilizes the previous lemma applied to a special family of intersections of $X$ with modular arithmetic progressions. Each set in this special family has size a power of $2$. We show that any set which is an intersection of $X$ with a modular arithmetic progression can be written as the union of two sets, and each of them is a set difference of two sets that each has a canonical decomposition into sets of different sizes from this special family. We obtain the lemma by putting these together and using the triangle inequality. 

	\begin{lemma}\label{lem:partial-coloring}
		Let $X\subseteq \Zn$ be a set of size $m > 0$. There exists a partial coloring $\chi: X\to \{-1, 0, 1\}$ that assigns $\pm 1$ to at least $m/10$ elements in $X$ such that
		\[\max_{A\in \mathcal{A}_n}|\chi(A\cap X)| \leq 200m^{\frac12}\left(\log \frac{en}{m}\right)^{\frac{1}{2}}.\]
	\end{lemma}
	\begin{proof}

		Let $\AP_n(a, d, i, j) := \{a+kd: i\leq k \leq j\}$ where $a, d\in \Zn$ and $i, j\in \Z$. By definition, 
		\[\CA_n = \left\{\AP_n(a, d, 0, l-1): a, d\in \Zn, 0\leq l \leq \frac{n}{\gcd(n, d)}\right\}.\] 
		
		Consider the following decomposition of sets in $\CA_n$. Let $$C_1 = \left\{\AP_n(a, d, i, j): 1\leq d < n, 0\leq a < \gcd(n, d), 0\leq i\leq j < \frac{n}{\gcd(n, d)}\right\}.$$ Now we show that any set $A\in \CA_n$ can be written as the disjoint union of at most two sets in $C_1$. For any set $A = \AP_n(a, d, 0, l-1)$ with $a, d\in \Zn$ and $0\leq l \leq \frac{n}{\gcd(n, d)}$, we may assume that $l > 0$, or otherwise it is an empty set. If $d = 0$, then the set is the singleton $\{a\}$, which is also $\AP_n(0, 1, a, a)\in C_1$ by choosing $i, j$ to be the integer representative of $a$ in the range $[0, n)$. Now we may assume $1 \leq d  < n$. Note that $\{kd: 0\leq k < \frac{n}{\gcd(n, d)}\}$ splits $\Zn$ into $\frac{n}{\gcd(n, d)}$ equal-sized intervals, so there is some $0\leq k < \frac{n}{\gcd(n, d)}$ such that $0\leq a - kd < \gcd(n, d)$. Define $a' = a - kd$, and we have
		\[\AP_n(a, d, 0, l - 1) = \AP_n(a', d, k, k + l - 1).\]
		Moreover we have $k < \frac{n}{\gcd(n, d)}$ and $l -1 < \frac{n}{\gcd(n, d)}$, so $0\leq k+l-1 < 2\frac{n}{\gcd(n, d)}$. If $k+l-1 < \frac{n}{\gcd(n, d)}$, then $\AP_n(a', d, k, k+l-1)\in C_1$. If $k + l-1\geq \frac{n}{\gcd(n, d)}$, then we may write
		\[\AP_n(a', d, k, k+l-1) = \AP_n\left(a', d, k, \frac{n}{\gcd(n, d)}-1\right) \cup \AP_n\left(a', d, 0, k+l -1- \frac{n}{\gcd(n, d)}\right).  \]
		
		Hence we know that for any partial coloring $\chi: \Zn \to \{-1, 0, 1\}$, we always have
		\begin{equation}\label{eqn:decomp-1}
		\max_{A\in \CA_n}|\chi(A\cap X)| \leq 2 \max_{A\in C_1}|\chi(A\cap X)|.
		\end{equation}
		
		Now we study the sets in $C_X := \{A\cap X: A\in C_1\}$. By definition we have
		\begin{equation}\label{eqn:decomp-1.5}
		\max_{A\in C_1}|\chi(A\cap X)| = \max_{B\in C_X}|\chi(B)|.
		\end{equation}
		For each $1\leq d < n$ and $0\leq a < \gcd(n, d)$, we define $A_{d, a} = \left\{a+kd: 0\leq k < \frac{n}{\gcd(d, n)}\right\}$, and define $X_{d, a} = A_{d, a}\cap X$. Since each element $x$ in $X_{d, a}$ is associated to an integer $0\leq k < \frac{n}{\gcd(d, n)}$ via the relation $x = a+kd$, we may order elements in $X_{d, a}$ in ascending order of $k$. We write $X_{d, a} = \left\{x_{d, a}^{(1)}, \dots, x_{d, a}^{(l_{d, a})}\right\}$, where $l_{d, a} = |X_{d, a}|$. For integers $i$ and $j$, we define $X_{d, a}(i, j) = \left\{x_{d, a}^{(t)}: i \leq t \leq j\right\}$.

		For each $d\in \Zn$, because each element of $\Zn$ is in exactly one $A_{d, a}$ for $0\leq a < \gcd(d, n)$, we have
		\begin{equation}\label{eqn:sum-of-l}\sum_{a = 0}^{\gcd(d, n) - 1}l_{d, a} = \sum_{a = 0}^{\gcd(d, n) - 1}|X\cap A_{d, a}| = |X| = m.\end{equation}
		
		For $A = \AP_n(a, d, i, j)$, if we take $i'$ to be the smallest index such that $x_{d, a}^{(i')} = a + k_1d$ with $k_1 \geq i$ (or $l_{d, a}+1$ if no such index with $0\leq k_1 < \frac{n}{\gcd(d, n)}$ exists), and $j'$ to be the largest index such that $x_{d, a}^{(j')} = a+k_2d$ with $k_2 \leq j$ (or $0$ if no such index with $0\leq k_2 < \frac{n}{\gcd(d, n)}$ exists), then we have $A\cap X = X_{d, a}(i', j')$. Hence each nonempty set in $C_X$ is of the form $X_{d, a}(i, j)$ with $1\leq i \leq j \leq l_{d, a}$. Therefore we may write
		\[C_X =  \{X_{d, a}(i, j): 1\leq d < n, 0\leq a < \gcd(n, d), 1\leq i\leq j \leq l_{d, a}\} \cup \{\emptyset\}.\]
		
		Now we further decompose sets in $C_X$ into canonical sets. Formally,
		\begin{equation}\label{eqn:define-C2}
		    C_2 := \left\{X_{d, a}(1+(t-1)2^i, t2^i): 1\leq d < n, 0\leq a < \gcd(n, d), 0\leq i, 1\leq t \leq \left\lfloor\frac{l_{d, a}}{2^i}\right\rfloor\right\}.
		\end{equation}
		Notice that each set $S\in C_X$ can be written as $U\setminus V$ with $V\subseteq U$ where both $U$ and $V$ can be written as the disjoint union of sets in $C_2$ of different sizes. This is trivially true for the empty set. For $S = X_{d, a}(i, j)$, we may take $U = X_{d, a}(1, j)$ and $V = X_{d, a}(1, i-1)$. If we write $j$ in binary form $2^{b_1} + 2^{b_2} + \dots + 2^{b_s}$ for $b_1 > b_2 > \dots > b_s$, then $U$ is the disjoint union of sets of the form $X_{d, a}(1+\sum_{k = 1}^{t-1}2^{b_k},  \sum_{k = 1}^{t}2^{b_k})$ for $t = 1, \dots, s$. We can decompose $V$ similarly.
		
		Note that sets in $C_2$ are all of size $2^i$ for $2^i \leq m$. If $\chi: X\to \{-1, 0, 1\}$ is such that for all $0\leq i\leq \log_2 m$ we have $\chi(S) \leq \Delta_i$ for any set $S\in C_2$ of size $2^i$, then by using the decomposition property above we have
		\begin{equation}\label{eqn:decomp-2}
		\max_{B\in C_X}|\chi(B)| \leq 2\sum_{0\leq i \leq \log_2 m} \Delta_i.
		\end{equation}
		
		For each $i$, we upper bound the number of sets in $C_2$ of size $2^i$ as following. For each $1\leq d < n$ and $0\leq a < \gcd(d, n)$, the number of choices of $t$ in \eqref{eqn:define-C2} is $\lfloor{l_{d, a}}/{2^i}\rfloor$. By \eqref{eqn:sum-of-l} the number of such sets is at most
		\begin{equation}\label{eqn:bound-num-of-sets}\sum_{d=1}^{n-1} \sum_{a = 0}^{\gcd(d, n) - 1} \left\lfloor\frac{l_{d, a}}{2^i}\right\rfloor \leq \sum_{d=1}^{n-1} \sum_{a = 0}^{\gcd(d, n) - 1} \frac{l_{d, a}}{2^i} = \frac{(n-1)m}{2^i}.\end{equation}
		
		Now we define $b(0) = 0$ and for $0 < s \leq n$,
		\begin{equation}\label{eqn:mod-n-prop-def-b}b(s) = 5\sqrt{s} \left(\log \frac{en}{s}\right)^{1/2}.\end{equation}
		From the definition, we know that $b(s) \geq 2\sqrt{s}$ for all $s\in \N$. We want to show that there exists a partial coloring $\chi$ that colors at least $m/10$ elements in $X$, and for any $S\in C_2$,
		\begin{equation}\label{eqn:mod-n-prop-target}|\chi(S)| \leq b(|S|).\end{equation}
		
		In order to apply Lemma~\ref{lem:partial-coloring-lemma} to $X$ and $C = C_2$, it remains to check that
		\begin{equation}\label{eqn:mod-n-prop-lemma-cond}\sum_{S\in C_2} \exp\left(-\frac{b(|S|)^2}{4|S|}\right) \leq m/50.\end{equation}
		From the definition of $b(\cdot)$ in \eqref{eqn:mod-n-prop-def-b}, we know that
		\begin{equation}\label{eqn:mod-n-prop-lemma-verify}\sum_{S\in C_2} \exp\left(-\frac{b(|S|)^2}{4|S|}\right) = \sum_{S\in C_2}e^{-\frac{25}{4}}\cdot \left(\frac{|S|}{n}\right)^{\frac{25}{4}} \leq e^{-6}\cdot\sum_{S\in C_2}\left(\frac{|S|}{n}\right)^2.\end{equation}
		Using the bound in \eqref{eqn:bound-num-of-sets}, we have
		\[\begin{split}
		\sum_{S\in C_2} \left(\frac{|S|}{n}\right)^{2} & \leq \sum_{0\leq i \leq \log_2 m} \frac{(n-1)m}{2^i} \left(\frac{2^i}{n}\right)^2 \leq 2m.
		\end{split}\]
		%
		Put this into \eqref{eqn:mod-n-prop-lemma-verify}. We get
		\[\sum_{S\in C_2} \exp\left(-\frac{b(|S|)^2}{4|S|}\right) \leq e^{-6}\cdot 2m \leq m/50.\]
		This shows that \eqref{eqn:mod-n-prop-lemma-cond} holds. By Lemma~\ref{lem:partial-coloring-lemma}, we conclude that there exists a partial coloring $\chi: X\to \{-1, 0, 1\}$ that assigns $\pm 1$ to at least $m/10$ elements, and \eqref{eqn:mod-n-prop-target} holds for any $S\in C_2$. Thus in \eqref{eqn:decomp-2} we may take $\Delta_i = b(2^i)$ for $0\leq i \leq \log_2 m$.
		
		It remains to show that $\chi$ satisfies the desired property. By \eqref{eqn:decomp-1}, \eqref{eqn:decomp-1.5}, and \eqref{eqn:decomp-2},
		\[\max_{A\in \mathcal{A}_n}|\chi(A\cap X)| \leq 4\sum_{0\leq i\leq \log_2 m}b(2^i) = 20 \sum_{i = 0}^{\lfloor{\log_2 m}\rfloor} 2^{i/2}\left( \log \frac{en}{2^i}\right)^{\frac12} \leq 200\cdot \sqrt{m}\left(\log \frac{en}{m}\right)^{\frac12}.\]
		To obtain the last inequality, observe that the sum is roughly a geometric series with ratio $2^{1/2}$ (although there is an extra logarithmic factor that slightly complicates this) and can be bounded by a constant factor times the largest term. A careful analysis gives the claimed bound. 
		
		Hence the partial coloring $\chi$ satisfies the desired inequality.
	\end{proof}
	
	
	\begin{corollary}\label{cor:subset-coloring}
		Let $X\subseteq \Zn$ be a set of size $m > 0$. There exists an absolute constant $c$ such that there is a coloring $\chi: X\to \{-1, 1\}$ satisfying
		\[\max_{A\in \mathcal{A}_n}|\chi(A\cap X)| \leq cm^{\frac12}\left(\log \frac{en}{m}\right)^{\frac 12}.\]
		In particular, for $X= \Zn$, we have 
		\[\disc(\mathcal{A}_n) = O(\sqrt{n}). \]
	\end{corollary}
	\begin{proof}
		The main idea is to iteratively apply Lemma~\ref{lem:partial-coloring} to the set of uncolored elements of $X$ until all elements of $X$ are colored.
		
		Start with $X_0 = X$. For each $i\geq 0$, we apply Lemma~\ref{lem:partial-coloring} to $X_i$ to get a partial coloring $\chi_i: X_i \to \{-1, 0, 1\}$, and we let $X_{i+1} = \chi_i^{-1}(0) \subseteq X_i$ to be the set of uncolored elements in $i$-th iteration. We continue this process until the $k$-th iteration where all elements are colored (i.e. $X_{k+1} = \emptyset$). Let $\chi$ be the final coloring given by $\chi(x_i) = \chi_i(x_i)$ if $x_i\in X_i\setminus X_{i+1}$ for $0\leq i\leq k$. We know that $|X_i| \leq 0.9|X_{i-1}|$ for all $1\leq i\leq k$, so $|X_i| \leq (0.9)^i|X_0| = (0.9)^im$. Moreover, we know that for any $A\in \mathcal A_n$, noting that $x^{1/2}\cdot \left(\log \frac{en}{x}\right)^{\frac 12}$ is monotonically increasing for real number $x \in (0, n]$,
		\[
		\begin{split}
		|\chi(A\cap X)| & = \left|\sum_{i=0}^k \chi_i(X_i\cap A)\right| \leq \sum_{i=0}^k \left|\chi_i(X_i\cap A)\right| < \sum_{i=0}^k 200\cdot |X_i|^{\frac12}\left(\log\frac{n}{e|X_i|}\right)^{\frac 12} \\
		& \leq 200m^{\frac12}\cdot \sum_{i=0}^k (0.9)^{\frac i2} \left(\log\frac{en}{(0.9)^im}\right)^{\frac 12} < cm^{\frac12}\left(\log\frac{en}{m}\right)^\frac 12
		\end{split}\]
		for some appropriate choice of the absolute constant $c$.
	\end{proof}

	\section{Upper Bounds}\label{section upper}
	By Corollary~\ref{cor:subset-coloring}, we have $\disc(\CA_n) = O(\sqrt{n})$. In this section, we prove the following upper bound on the same quantity, which shows that the previous bound is not always tight, and that $\disc(\CA_n)$ depends on the factorization of $n$. Recall that $\omega(r)$ denotes the number of distinct prime factors of $r$.

	\begin{theorem}\label{thm:bound-from-distribution}
		There exists an absolute constant $c> 0$ such that for any positive integer $n$, we have 
		\[\disc(\CA_n) \leq \min_{r|n}\left(\frac{n}{r} + c\sqrt{r}\cdot 2^{\omega(r)}\right).\]
	\end{theorem}
	
	The proof of Theorem \ref{thm:bound-from-distribution} consists of two steps. The first key step is to establish an upper bound on $\disc(\CA_n)$ using a coloring of $\Z_r$ where $r$ divides $n$ that simultaneously has low discrepancy with respect to two particular families of arithmetic progressions in $\Z_r$. This is achieved by Lemma~\ref{lem:copies-of-same-coloring}. 
	To accomplish this, we need to introduce a special type of arithmetic progression in $\Zn$. Recall that $\CA_n$ is the set of all arithmetic progressions in $\Zn$. We define $\CA_n^0$ to be the set of all congruence classes of $\Zn$. Formally, for any $r, i\in \Zn$, let
	\[C(r, i) := \{x\in \Zn: x\equiv i \Mod r\}\]
	 and
	\[\CA_n^0 := \{C(r, i): r, i\in \Zn\}.\]
	In particular we know that $\CA_n^0 \subseteq \CA_n$. In Lemma~\ref{lem:copies-of-same-coloring} we obtain an upper bound on $\disc(\CA_n)$ if there is an $r$ that divides $n$ and a coloring $\chi$ of $\Z_r$ that has low discrepancy over both $\CA_r$ and $\CA_r^0$.
	
	The second step is to find a coloring of $\Z_r$ which has nearly optimal discrepancy over $\CA_r$ and $\CA_r^{0}$ simultaneously. Finally we will complete the proof by applying Lemma~\ref{lem:copies-of-same-coloring} with this coloring of $\Z_r$.

	\begin{lemma}\label{lem:copies-of-same-coloring}
		Let $n$ be a positive integer and $r$ be a positive factor of $n$. For any $\chi: \Z_r\to \{1, -1\}$,
		\[\disc(\CA_{n}) \leq \max_{A\in \CA_r}|\chi(A)| + \frac{n}{r}\cdot \max_{A_0\in \CA_r^0}|\chi(A_0)|.\]
	\end{lemma}
	\begin{proof}
		Let $m = n/r$. For the subgroup $(m) \subseteq \Zn$, we have a quotient map $\tau: \Zn \to \Zn/(m) = \Z_r$ given by $\tau(x)= \overline{x} = x+(m)$. Suppose we are given an arbitrary $\chi: \Z_r \to \{1, -1\}$. Now we define $\chi' = \chi \circ \tau: \Zn\to \{1, -1\}$. It suffices to show
		\begin{equation}\label{eqn:copies-of-same-coloring}\max_{A'\in \CA_{n}} |\chi'(A')| \leq  \max_{A\in \CA_r}|\chi(A)| + m\cdot \max_{A_0\in \CA_r^0}|\chi(A_0)|.\end{equation}
		
		Let $A' = \{a + kd: 0\leq k < L\}$ for some $a, d\in \Z_{n}$ be an arithmetic progression in $\Z_{n}$ of length $L$. We may assume that $L \leq n/{\gcd(d, n)}$ so that there is no repeated element in $A'$.
		
		Let $L_0 = r/\gcd(r, \overline{d})$ where $\overline{d} = \tau(d)$. For any integer $t > 0$, we know that $\{\overline{a} + k\overline{d}: t \leq k < t+L_0\}$ is a set of size $L_0$ in $\CA^0_n$, as it covers each element in the congruence class exactly once. For any $0\leq l < L_0$, as the set $\{\overline{a} + k\overline{d}: t \leq k < t+l\}$ has no repeated element, it is a set of size $l$ in $\CA_n$.
		
		Now we may write $L = qL_0 + r$ for some integers $q$ and $0\leq r < L_0$. We know that
		\[q\leq \frac{L}{L_0} \leq \frac{n/\gcd(d, n)}{r/\gcd(r, \overline{d})} = m\cdot \frac{\gcd(r, d)}{\gcd(n, d)} \leq m.\]
		Applying the triangle inequality, we get
		\[\begin{split}
		|\chi'(A')| & = \left|\sum_{k = 0}^{L-1}\chi'(a+kd)\right| = \left|\sum_{k = 0}^{L-1}\chi(\overline{a}+k\overline{d})\right| =  \left|\sum_{t = 0}^{q-1}\sum_{k = tL_0}^{(t+1)L_0-1}\chi(\overline{a}+k\overline{d}) + \sum_{k = qL_0}^{qL_0 + r-1}\chi(\overline{a}+k\overline{d})\right| \\
		& \leq \sum_{t = 0}^{q-1}\left|\sum_{k = tL_0}^{(t+1)L_0-1}\chi(\overline{a}+k\overline{d})\right| + \left|\sum_{k = qL_0}^{qL_0 + r-1}\chi(\overline{a}+k\overline{d})\right| \\
		& \leq q \cdot \max_{A_0\in \CA_r^0}|\chi(A_0)| + \max_{A\in \CA_r}|\chi(A)| \leq m\cdot \max_{A_0\in \CA_r^0}|\chi(A_0)| + \max_{A\in \CA_r}|\chi(A)|.
		\end{split}\]
		Since the bound is uniform for all $A'\in \CA_{n}$, we conclude that \eqref{eqn:copies-of-same-coloring} holds.
	\end{proof}
	
	After the above deduction, the second step is to find a coloring $\chi: \Zn \to \{1, -1\}$ that has nearly optimal discrepancy over $\CA_n$ and $\CA_n^0$ simultaneously. Formally we have the following statement.
	\begin{lemma}\label{lem:color-with-bounds-for-cong-classes}
		There exists an absolute constant $c > 0$ such that the following holds. Suppose that $n$ is a positive integer with $k$ distinct prime factors. Then there exists $\chi: \Zn \to \{1, -1\}$ such that
		\[\max_{A_0\in \CA_n^0}|\chi(A_0)| \leq 1\quad \AND \quad  \max_{A\in \CA_n}|\chi(A)| \leq c\sqrt{n}\cdot 2^{k}.\]
	\end{lemma}
	
	We can now deduce Theorem~\ref{thm:bound-from-distribution} from the above two steps. 
	\begin{proof}[Proof of Theorem~\ref{thm:bound-from-distribution} assuming Lemma~\ref{lem:color-with-bounds-for-cong-classes}]
		For any positive factor $r$ of $n$, by Lemma~\ref{lem:color-with-bounds-for-cong-classes}, there is a coloring $\chi$ of $\Z_r$ such that
		\[\max_{A_0\in \CA_r^0}|\chi(A_0)| \leq 1 \quad \AND  \quad  \max_{A\in \CA_r}|\chi(A)| \leq c\sqrt{r}\cdot 2^{\omega(r)},\]
		where $c$ is an absolute constant. Then applying Lemma~\ref{lem:copies-of-same-coloring}, we conclude that
		\[\disc(\CA_n) \leq \frac{n}{r} + c\sqrt{r}\cdot 2^{\omega(r)}.\]
		As this holds for all positive factors $r$ of $n$, the desired result follows.
	\end{proof}
	
	It remains to prove Lemma~\ref{lem:color-with-bounds-for-cong-classes}.
	We first present a proof of the case that $n$ is a prime power. The proof of this case is simpler and illustrates the main ideas in the proof of Lemma~\ref{lem:color-with-bounds-for-cong-classes}.
	\begin{lemma}\label{lem:color-with-bounds-for-cong-classes-dim-1}
		There exists absolute constant $c > 0$ such that the following holds. Let $n = p^\alpha$ for some prime $p$ and positive integer $\alpha$. Then there exists $\chi: \Zn \to \{1, -1\}$ such that 
		\[\max_{A_0\in \CA_n^0}|\chi(A_0)| \leq 1 \quad \AND \quad  \max_{A\in \CA_n}|\chi(A)| \leq c\sqrt{n}.\]
	\end{lemma}
	To find a coloring with low discrepancy over $\CA_n^{0}$, 
	the main idea is to color simultaneously two disjoint subsets $S_1$ and $S_2$ of $\Zn$ that are translations of each other, and we hope that $\chi(S_1\cap A_0)$ and $\chi(S_2\cap A_0)$ cancel out for many $A_0\in \CA_n^0$. The $c\sqrt{n}$ upper bound on the discrepancy over $\CA_n$ in Lemma~\ref{lem:color-with-bounds-for-cong-classes-dim-1} is achieved via Corollary~\ref{cor:subset-coloring}. We first derive the following lemma that gives a desired coloring for a subset $X$ of $\Zn$, provided that $X$ is an initial interval of some special length.
	
	\begin{lemma}\label{lem:partitioned-coloring-dim-1}
		There exists absolute constant $c > 0$ such that the following holds. Let $n = p^\alpha$ for some prime $p$ and positive integer $\alpha$. Suppose that $X = \{0, \dots, m-1\}\subseteq \Zn$ for some $m\leq n$ satisfies that $m = sp^\beta$ with $s$ even and $\beta$ being a nonnegative integer. Then there exists $\chi: X\to \{1,-1\}$ such that
		\[\max_{A\in \CA_n}|\chi(A\cap X)| \leq c\sqrt{m}\left(\log\frac{en}{m}\right)^{\frac 12},\]
		and for any $w\in \Zn$ and $r = p^\gamma$ with $0\leq \gamma \leq \beta$,
		\[\chi(C(r, w)\cap X) = 0.\]
	\end{lemma}
	\begin{proof}
		Let $S = \{0, \dots, m/2-1\}$. By Corollary~\ref{cor:subset-coloring}, there is $\chi_0:X\to \{1, -1\}$ such that 
		\[|\chi_0(A\cap S)| \leq c_0\sqrt{\frac{m}{2}}\left(\log \frac{2en}{m}\right)^{\frac 12},\]
		holds for any $A\in \CA_n$, where $c_0$ is an absolute constant. We know that $X$ is the disjoint union of $S$ and $S+\frac{m}{2}$, i.e.~each element $x\in X$ is of the form $x = s + v\frac{m}{2}$ for some $s\in S$ and $v\in \{0, 1\}$. We define $\chi: X \to \{1, -1\}$ given by $\chi(s+v\frac{m}{2}) = (-1)^v\cdot \chi_0(s)$ for any $s\in S$ and $v\in \{0, 1\}$.
		
		For any $A\in \CA_n$, noting that $A - \frac{m}{2}$ is also an arithmetic progression,
		\[\begin{split}
		|\chi(A\cap X)| & = |\chi(A\cap S) + \chi(A\cap (S+m/2))| = |\chi_0(A\cap S) + (-1)\cdot \chi_0((A-m/2)\cap S)|\\
		& \leq  |\chi_0(A\cap S)| + |\chi_0((A-m/2)\cap S)| \leq 2\cdot c_0\sqrt{\frac{m}{2}}\left(\log \frac{2en}{m}\right)^{\frac 12} \leq c\sqrt{m}\left(\log \frac{en}{m}\right)^{\frac 12},
		\end{split}\]
		where $c$ is an absolute constant. Next we verify that for any $r = p^{\gamma}$ with $\gamma \leq \beta$ and any $w\in \Zn$, $\chi(C(r, w) \cap X) = 0$. Note that $m/2$ is a multiple of $r$, so $C(r, w) - m/2 = C(r, w)$. Hence we have
		\[\chi(C(r, w)\cap X) = \chi_0(C(r, w)\cap S) + (-1)\cdot \chi_0((C(r, w)-m/2)\cap S) =\chi_0(C(r, w)\cap S) - \chi_0(C(r, w)\cap S) = 0. \]
		We have thus shown the existence of such a coloring $\chi$ with the desired properties.
	\end{proof}
	
	In the case $n = p^{\alpha}$, we partition $\Zn$ into a few sets, each being an interval of $\Zn$ of length an even number times a power of $p$ as in Lemma~\ref{lem:partitioned-coloring-dim-1}.
	\begin{proof}[Proof of Lemma~\ref{lem:color-with-bounds-for-cong-classes-dim-1}]
		Note that for any $C(r, w)\in \CA_n^0$, we have $C(r, w) = C(\gcd(r, n), w)$. Hence it suffices to check $|\chi(C(r, w))| \leq 1$ where $r$ divides $n = p^\alpha$, i.e., $r = p^\gamma$ for some nonnegative $\gamma \leq \alpha$.
		
		Let $c_0$ be the constant in Lemma~\ref{lem:partitioned-coloring-dim-1}. 
		
		If $p = 2$, we can take $m = n = 2\cdot 2^{\alpha-1}$ in Lemma~\ref{lem:partitioned-coloring-dim-1}. Thus, there exists $\chi: \Zn\to\{1, -1\}$ such that 
		\[\max_{A\in \CA_n} |\chi(A)| \leq c_0\sqrt{n},\]
		and that for any $w\in \Zn$ and any $r = 2^{\gamma}$ with $\gamma \leq \alpha-1$, $\chi(C(r, w)) = 0$. This means that, for any $r = 2^\gamma$, if $\gamma \leq \alpha-1$, then $\chi(C(r, w)) = 0$; if $\gamma = \alpha$, then $C(r, w)$ contains a single element $w$, so $|\chi(C(r, w))| = 1$ in this case. We conclude that
		\[\max_{A_0\in \CA_n^0}|\chi(A_0)| \leq 1.\]
		
		If $p > 2$ is an odd prime, then we partition $\Zn$ into a few intervals: $X_0 = \{0\}, X_i = \{x: p^{i-1} \leq x < p^i\}$ for $1\leq i\leq \alpha$. Note that for each $i\geq 1$, $X_i$ is an interval of length $(p-1)\cdot p^{i-1}$. We would like to apply Lemma~\ref{lem:partitioned-coloring-dim-1} to each $X_i$ for $i\geq 1$. It is applicable because the property is maintained by translation. Therefore for each $1\leq i\leq \alpha$ there is a coloring $\chi_i: X_i\to \{1, -1\}$ such that
		\[\max_{A\in \CA_n} |\chi_i(A\cap X_i)| \leq c_0\sqrt{(p-1)\cdot p^{i-1}} \left(1 + \log\frac{p^\alpha}{(p-1)\cdot p^{i-1}}\right)^{\frac{1}{2}},\]
		and for each $w\in \Zn$ and $r = p^\gamma$ with $\gamma \leq i-1$, we have $\chi_i(C(r, w)\cap X_i) = 0.$
		
		For $X_0$ we may take any arbitrary $\chi_0: X_0 \to \{1, -1\}$. Note that $X_0 = \{0\}$ contains only one element, so $|\chi_0(A\cap X_0)| \leq 1$ for all $A\subseteq \Zn$.
		
		Since $X = \coprod_{i = 0}^{\alpha}X_i$, we define $\chi: X\to \{1, -1\}$ such that $\chi(x_i) = \chi_i(x_i)$ for all $0\leq i\leq \alpha$ and $x_i\in X_i$.
		Consequently, we know that for each $A \in \CA_n$,
		\[\begin{split}
		|\chi(A)| = \sum_{i = 0}^\alpha |\chi_i(A\cap X_i)| \leq 1 + \sum_{i = 1}^\alpha c_0\sqrt{(p-1)\cdot p^{i-1}} \left(1 + \log\frac{p^\alpha}{(p-1)\cdot p^{i-1}}\right)^{\frac{1}{2}} \leq c\sqrt{n}
		\end{split}\]
		for some appropriate constant $c$. Moreover, for each $C(r, w) \in \CA_n^0$, we may assume that $r = p^\gamma$ for some $0\leq \gamma\leq \alpha$ and it gives that
		\[\chi(C(r, w)) = \chi\left(C(r, w) \cap \left(\bigcup_{i=0}^{\gamma} X_i\right)\right) + \sum_{i  = \gamma+1}^{\alpha}\chi_i(C(r, w)\cap X_i).\]
		Now, for each $i \geq \gamma+1$, i.e. $\gamma \leq i-1$, we have $\chi_i(C(r, w)\cap X_i) = 0$. Also note that $\bigcup_{i=0}^{\gamma} X_i$ is $\{x: 0\leq x < p^\gamma\}$, so there is exactly one element of $C(r, w)$ in it. Hence
		\[|\chi(C(r, w))| = \left|\chi\left(C(r, w) \cap \left(\bigcup_{i=0}^{\gamma} X_i\right)\right)\right| = 1.\]
		Thus we may conclude that
		\[\max_{A_0\in \CA_n^0}|\chi(A_0)| \leq 1.\]
		Hence for both $p = 2$ and $p$ is odd, we can always find such a coloring $\chi$.
	\end{proof}
	
	Finally we extend the argument above to arbitrary $n$ using Chinese remainder theorem.
	\begin{notations}
		Suppose that number $n$ has prime factorization $n = p_1^{\alpha_1}\cdots p_k^{\alpha_k}$. By Chinese remainder theorem, we have an isomorphism
		\[\psi_n: \Z_{p_1^{\alpha_1}} \times \cdots \times \Z_{p_k^{\alpha_k}} \to \Zn.\]
	\end{notations}

	Again, we first show the following analogue of Lemma~\ref{lem:partitioned-coloring-dim-1}. This shows that if $X\subseteq \Zn$ has some special structure, then we have a coloring of $X$ with low discrepancy with respect to both $\CA_n$ and $\CA_n^0$.
	\begin{lemma}\label{lem:partitioned-coloring-dim-k}
		Let $n$ be a positive integer with prime factorization $n = p_1^{\alpha_1}\cdots p_k^{\alpha_k}$. For each $1\leq i\leq k$, let $T_i \leq p_i^{\alpha_i}$ be a positive integer. Let $X = \{\psi_n(t_1, \dots, t_k): 0\leq t_i < T_i \Forall 1\leq i\leq k\}$. Then $X\subseteq \Zn$ is a set of $m = \prod_{i=1}^k T_i$ elements. Suppose that  $I\subseteq [k]$ is a subset of indices and $(\beta_i)_{i\in I}$ is a sequence of nonnegative integers, such that for each $i\in I$, $T_i = s_ip_i^{\beta_i}$ for some even number $s_i$. Then there exists a coloring $\chi: X\to \{-1, 1\}$ such that the following holds. Its discrepancy over $\CA_n$ satisfies that
		\begin{equation}\label{eqn:lem:partitioned-coloring-dim-k-1}\max_{A\in \CA_n} |\chi(A\cap X)| \leq c2^{\frac{|I|}{2}}\cdot \sqrt{m}\left(|I|+1+\log\frac{n}{m}\right)^{\frac{1}{2}},\end{equation}
		and for any $w\in \Zn$ and any $i\in I$,
		\begin{equation}\label{eqn:lem:partitioned-coloring-dim-k-2}\chi(C(n/p_i^{\alpha_i - \beta_i}, w)\cap X) = 0.\end{equation}
	\end{lemma}
	\begin{proof}
		Let $L = |I|$. Let $S_i = T_i/2$ if $i\in I$, and $S_i = T_i$ otherwise. For each $i\in I$, from the condition that $T_i/p_i^{\beta_i}$ is even, we know that $S_i$ is a multiple of $p_i^{\beta_i}$. We define
		\[X_0 = \{\psi_n(t_1, \dots, t_k): 0\leq t_i < S_i \Forall 1\leq i\leq k\}.\]
		We know that $X_0\subseteq \Zn$ is a set of size $m/2^L$. Next we apply Corollary~\ref{cor:subset-coloring} to get the constant $c_0>0$ and a partial coloring of $\chi_0: X_0 \to \{-1, 1\}$ such that
		\[\max_{A\in \CA_n} |\chi_0(X_0\cap A)| \leq c_0\sqrt{\frac{m}{2^{L}}}\left(1+\log\frac{2^{L}n}{m}\right)^{\frac 12} \leq c\sqrt{\frac{m}{2^{L}}} \left(L+1+\log\frac{n}{m}\right)^{\frac12},\]
		where $c$ is another absolute constant.
		
		For each $v = (v_i)_{i\in I} \in \{0, 1\}^{I}$ (binary tuples of length $L$ indexed by $I$), we define $\sgn(v) = (-1)^{\sum_{i\in I}v_i}$, and let $u_v\in \Zn$ be the unique element, by the Chinese remainder theorem, such that 
		\[u_v\equiv v_i S_i\Mod{p_i^{\alpha_i}} \Forall i\in I \quad \AND \qquad u_v \equiv 0 \Mod{p_j^{\alpha_j}} \Forall j\notin I.\]
		Thus we have the decomposition of $X$ into disjoint copies of translations of $X_0$:
		\[X = \coprod_{v\in \{0, 1\}^I} X_v, ~~\text{where} \quad X_v := u_v+X_0.\]
		This implies that each $x\in X$ can be written in the form $x = u_v + x_0$ for some $v\in \{0, 1\}^I$ and $x_0\in X_0$. Now we define $\chi: X\to \{1, -1\}$ by $\chi(u_v+x_0) = \sgn(v)\cdot \chi(x_0)$ for all $v\in \{0, 1\}^I, x_0\in X_0$.
		For any $A\in \CA_n$, noting that $A- u_v$ is also in $\CA_n$, so
		\[\begin{split}|\chi(X\cap A)| & = \left|\sum_{v\in \{0, 1\}^I} \chi(X_v\cap A)\right| = \left|\sum_{v\in \{0, 1\}^I} \sgn(v)\cdot \chi_0(X_v\cap A - u_v)\right| \\
		& \leq \sum_{v\in \{0, 1\}^I} \left|\chi_0(X_0\cap (A - u_v))\right| \leq 2^L\max_{A\in \CA_n} |\chi_0(X_0\cap A)| \\
		& \leq 2^L \cdot c\sqrt{\frac{m}{2^L}} \left(L+1+\log\frac{n}{m}\right)^{\frac12} = c2^{\frac{L}{2}}\cdot \sqrt{m}\left(L+1+\log\frac{n}{m}\right)^{\frac{1}{2}}. \end{split}\]
		This shows that the coloring $\chi$ satisfies  the condition \eqref{eqn:lem:partitioned-coloring-dim-k-1} on the discrepancy over $\CA_n$.

		It remains to show \eqref{eqn:lem:partitioned-coloring-dim-k-2}. For each $i\in I$, let $r_i = n/p_i^{\alpha_i - \beta_i} = p_1^{\alpha_1}\cdots p_i^{\beta_i} \cdots p_k^{\alpha_k}$.
		For any $i\in I$ and any $w\in \Zn$, we know that $C(r_i, w)$ contains $p^{\alpha_i - \beta_i}$ elements. If $C(r_i, w)\cap X = \emptyset$, then \eqref{eqn:lem:partitioned-coloring-dim-k-2} trivially holds. It suffices to check $\chi(C(r_i, w)\cap X) = 0$ for $w = \psi_n(t_1,\dots, t_k)\in X_v$ for some $v\in \{0, 1\}^I$.  Since $p_j^{\alpha_j}$ divides $r_i$, each element in $C(r_i, w)$ is congruent to $t_j$ mod $p_j^{\alpha_j}$ for $j\ne i$. This implies that $C(r_i, w)$ only has nonempty intersection with $X_v$ and $X_{v'}$ where $v, v'\in \{0, 1\}^I$ differ only at the entry indexed by $i$. Moreover by the definition of $u$, $u_v - u_{v'}$ is congruent to $0$ mod $p_j^{\alpha_j}$ for all $j\ne i$, and is congruent to $\pm \frac{s_i}{2}p_i^{\beta_i}$ mod $p_i^{\alpha_i}$. Thus $u_v - u_{v'}$ is a multiple of $r_i$, which implies $C(r_i, w) = C(r_i, w) + (u_v - u_{v'})$, or equivalently $C(r_i, w) - u_v = C(r_i, w) - u_{v'}$. Therefore we have
		\[\begin{split}
		\chi(C(r_i, w)\cap X) & = \chi(C(r_i, w)\cap X_v) + \chi(C(r_i, w)\cap X_{v'}) \\
		& = \sgn(v) \cdot \chi_0((C(r_i, w) - u_v)\cap X_0) + \sgn(v') \cdot \chi_0((C(r_i, w) - u_{v'})\cap X_0) = 0.
		\end{split}\]
		This is true for all $i\in I$ and $w\in \Zn$, and thus we conclude that $\chi$ satisfies \eqref{eqn:lem:partitioned-coloring-dim-k-2} as well.
	\end{proof}

	We are now ready to prove Lemma~\ref{lem:color-with-bounds-for-cong-classes}. The idea is that, we partition $\Zn$ into a few subsets, one for each factor $r$ of $n$, and for each of them we use Lemma~\ref{lem:partitioned-coloring-dim-k} to get a coloring of this subset with low discrepancy over $\CA_n$ and nearly optimal discrepancy over $\CA_n^0$. Finally we show that the full coloring we get from combining these subset colorings satisfies the properties.
	\begin{proof}[Proof of Lemma~\ref{lem:color-with-bounds-for-cong-classes}]
		Suppose that $n$ has prime factorization $n = p_1^{\alpha_1}\cdots p_k^{\alpha_k}$. For each $1\leq i\leq k$, we partition $\Z_{p_i^{\alpha_i}}$ into $\alpha_i+1$ intervals $\left(S^{(t)}_{i}\right)_{t=0}^{\alpha_i}$ in the following way. If $p_i = 2$, then we set $S_i^{(\alpha_i)} = \Z_{p_i^{\alpha_i}}$ and $S_i^{(t)}  = \emptyset$ for all $0\leq t < \alpha_i$. If $p_i > 2$, then we set $S_i^{(0)} = \{0\}\subseteq \Z_{p_i^{\alpha_i}}$ and $S_i^{(t)} = \{x:p_i^{t-1}\leq x < p_i^t\}$ for $1\leq t\leq \alpha_i$. In particular we have $|S^{(t)}_i| \leq p_i^{t}$ for all $i$ and $t$. Moreover, for each $i$ and $t > 0$, $S_i^{(t)}$ is always an interval of length $s_ip_i^{t-1}$ in $\Z_{p_i^{\alpha_i}}$ for some even $s_i$.
		
		For each factor $r = p_1^{\delta_1}\cdots p_k^{\delta_k}$ of $n$, we define $Y_r = S_1^{(\delta_1)}\times \cdots \times S_k^{(\delta_k)}\subseteq \Z_{p_1^{\alpha_1}}\times \cdots \times \Z_{p_k^{\alpha_k}}$ and $X_r = \psi_n(Y_r)  \subseteq \Zn$.
		Since each $\Z_{p_i^{\alpha_i}}$ is partitioned into $S_i^{(t)}$ for $0\leq t \leq \alpha_i$, their product is partitioned into $Y_r$ for $r$ divides $n$, so $\Zn$ is partitioned into $X_r$ for $r$ divides $n$. Because $|S^{(t)}_i| \leq p_i^{t}$ for all $i$ and $t$,
		\begin{equation}\label{eqn:upper-bound-on-Xr}
		|X_r| = |Y_r| = \prod_{i = 1}^k |S_i^{(\delta_i)}| \leq \prod_{i = 1}^k p_i^{\delta_i} = r.
		\end{equation}
		For each nonempty $X_r$ with $r = p_1^{\delta_1}\cdots p_k^{\delta_k}$, we would like to apply Lemma~\ref{lem:partitioned-coloring-dim-k} to get a coloring of $X_r$ with bounded discrepancy. Note that we can translate $Y_r$ so that all intervals $S_j^{(\delta_j)}$ for $1\leq j\leq k$ start at $0$, and notice those properties from Lemma~\ref{lem:partitioned-coloring-dim-k} are invariant under translation, so it is applicable to $X_r$. Let $I_r = \{i\in [k]: \delta_i > 0\}\subseteq [k]$. We have that for each $i\in I_r$, $S_i^{(\delta_i)}$ is an interval of length $s_ip_i^{\beta_i}$ for some even $s_i$ and $\beta_i = \delta_i - 1$. This 
		implies that we can find a partial coloring $\chi_r: X_r\to \{-1, 1\}$ such that
		\begin{equation}\label{eqn:cong-classes-disc-over-AP}
		\begin{split}\max_{A\in \CA_n}|\chi_r(A\cap X_r)| & \leq c_02^{\frac{|I_r|}{2}}\sqrt{|X_r|}\left(|I_r| + 1 + \log\frac{n}{|X_r|}\right)^{\frac{1}{2}},
		\end{split}
		\end{equation}
		and for any $i\in I_r$ and any $w\in \Zn$,
		\begin{equation}\label{eqn:cong-classes-disc-over-cong-classes}
		\chi_r(C(n/p_i^{\alpha_i - \beta_i}, w)\cap X_r) = \chi_r(C(n/p_i^{\alpha_i - \delta_i+1}, w)\cap X_r) = 0.
		\end{equation}
		
		Having defined $\chi_r$ as above for each nonempty $X_r$, we now construct $\chi$ by setting $\chi(x_r) = \chi_r(x_r)$ for each $x_r\in X_r$ for all $r$ that divides $n$. We first show that $\chi$ has low discrepancy over $\CA_n$. By \eqref{eqn:upper-bound-on-Xr} we have $|X_r| \leq r$. By definition we have $|I_r| \leq k$. Also note that $x^{1/2}\cdot \left(k+1+\log \frac{n}{x}\right)^{\frac{1}{2}}$ is increasing for $x\in (0, n]$. Combining these together, we know that for any nonempty $X_r$,
		\begin{equation}\label{eqn: size estimation}
		2^{\frac{|I_r|}{2}}\sqrt{|X_r|}\left(|I_r| + 1 + \log\frac{n}{|X_r|}\right)^{\frac{1}{2}} \leq  2^\frac{k}{2}\sqrt{r}\left(k+ 1 + \log\frac{n}{r}\right)^{\frac{1}{2}} \leq 2^{\frac k2}(k+1)^{\frac 12} \sqrt{r}\left(\log\frac{en}{r}\right)^{\frac 12}.
		\end{equation}
		Combining \eqref{eqn:cong-classes-disc-over-AP} and \eqref{eqn: size estimation},  there exists an absolute constant $c_1$ such that for each nonempty $X_r$,
		\begin{equation}\label{eqn:cong-classes-disc-over-AP-1}
		\max_{A\in \CA_n}|\chi_r(A\cap X_r)| \leq c_02^{\frac{k}{2}}(k+1)^{\frac12}\cdot \sqrt{r}\left(\log\frac{en}{r}\right)^{\frac{1}{2}} \leq c_12^{\frac{3}{4}k}\cdot \sqrt{r}\left(1 + \log\frac{n}{r}\right)^{\frac{1}{2}}.
		\end{equation}
		This is also true for $X_r = \emptyset$. Moreover we know that for any $r = p_1^{\delta_1}\cdots p_k^{\delta_k}$,
		\[ \sqrt{r}\left(1 + \log\frac{n}{r}\right)^{\frac{1}{2}} \leq \prod_{i = 1}^k p_i^{\delta_i/2} \left(1 + \log\frac{p_i^{\alpha_i}}{p_i^{\delta_i}}\right)^{\frac{1}{2}} = \sqrt{n} \prod_{i = 1}^k p_i^{-(\alpha_i - \delta_i)/2} \left(1 + (\alpha_i - \delta_i)\log p_i\right)^{\frac{1}{2}}.\]
		Summing this over all factors $r = p_1^{\delta_1}\cdots p_k^{\delta_k}$ for $0 \leq \delta_i\leq \alpha_i$, we have that if we set $t_i = \alpha_i - \delta_i$, 
		\[\begin{split}\sum_{r|n}\sqrt{r}\left(1 + \log\frac{n}{r}\right)^{\frac{1}{2}} &  \leq  \sqrt{n} \cdot \sum_{\delta_1 = 0}^{\alpha_1}\cdots \sum_{\delta_k = 0}^{\alpha_k}\prod_{i = 1}^k p_i^{-(\alpha_i - \delta_i)/2} \left(1 + (\alpha_i - \delta_i)\log p_i\right)^{\frac{1}{2}} \\
		& = \sqrt{n} \prod_{i = 1}^k \left(\sum_{t_i = 0}^{\alpha_i}p_i^{-t_i/2}(1+t_i\log p_i)^{\frac12}\right) \leq  \sqrt{n} \prod_{i = 1}^k \left(\sum_{t_i = 0}^{\infty}p_i^{-t_i/2}(1+t_i\log p_i)^{\frac12}\right).
		\end{split}\]
		Note that for $p \geq 2$,
		\[f(p) := \sum_{t = 0}^{\infty}p^{-t/2}(1+t\log p)^{\frac12}\]
		is well-defined and tends to $1$ as $p \to \infty$. Thus there exists an absolute constant $P > 0$ such that if $p \geq P$, then $f(p) < 2^{\frac{1}{4}}$. Since there are only finitely many prime numbers less than $P$, we conclude that there exists an absolute constant $c_2$ such that
		\[\prod_{i = 1}^k \left(\sum_{t_i = 0}^{\infty}p_i^{-t_i/2}(1+t_i\log p_i)^{\frac{1}{2}}\right) = \prod_{i = 1}^k f(p_i) \leq c_2 \cdot 2^{\frac{k}{4}}.\]
		Thus we have
		\begin{equation}\label{eqn:cong-classes-disc-over-AP-2}
		\sum_{r|n}\sqrt{r}\left(1 + \log\frac{n}{r}\right)^{\frac{1}{2}} \leq \sqrt{n}\prod_{i = 1}^k \left(\sum_{t_i = 0}^{\infty}p_i^{-t_i/2}(1+t_i\log p_i)^{\frac{1}{2}}\right) \leq c_22^{\frac{k}{4}}\sqrt{n}.
		\end{equation}
		Thus for any $A\in \CA_n$, by \eqref{eqn:cong-classes-disc-over-AP-1} and \eqref{eqn:cong-classes-disc-over-AP-2},
		\[|\chi(A)| = \left|\sum_{r|n}\chi_r(A\cap X_r)\right| \leq \sum_{r|n}c_12^{\frac{3}{4}k}\cdot \sqrt{r}\left(1 + \log\frac{n}{r}\right)^{\frac{1}{2}} \leq c_1c_22^{k}\sqrt{n}.\]
		Here we set $c = c_1c_2$ to get the desired bound on the discrepancy over $\CA_n$.
		
		Now we study the discrepancy of $\chi$ over $\CA_n^0$. Note that $C(r', w) = C(\gcd(r', n), w)$. We may only consider the set of $C(r', w) \in \CA_n^0$ for which $r' = p_1^{\gamma_1}\cdots p_k^{\gamma_k}$ is a factor of $n$. By definition of $\chi$, we have the relation that
		\begin{equation}\label{eqn:cong-classes-disc-as-sum}\chi(C(r', w)) = \sum_{r|n}\chi_r(C(r', w)\cap X_r).\end{equation}
		If $r'$ is not a multiple of $r = p_1^{\delta_1}\cdots p_k^{\delta_k}$, then we know that there exists some $1\leq i\leq k$ for which $\gamma_i \leq \delta_i - 1$. It follows that $\delta_i > 0$, and $i\in I_r$. Since $r'$ is a factor of $n$ and $\gamma_i \leq \delta_i - 1$, it is a factor of $n/p_i^{\alpha_i - \delta_i + 1}$. Consequently, the congruence class $C(r', w)$ can be partitioned into congruence classes of the form $C(n/p_i^{\alpha_i - \delta_i + 1}, w')$. Noting that \eqref{eqn:cong-classes-disc-over-cong-classes} holds for any $w'\in \Zn$, we know that if $r$ is not a factor of $r'$, then $\chi_r(C(r', w)\cap X_r) = 0$.
		Thus we can remove all summands on the right hand side of \eqref{eqn:cong-classes-disc-as-sum} except those $r$ that divides $r'$. We get
		\[\chi(C(r', w)) = \sum_{r|r'}\chi_r(C(r', w)\cap X_r) = \chi\left(C(r', w) \cap \left(\bigcup_{r|r'}X_r\right)\right).\]
		Note that we are taking union over all $r = p_1^{\delta_1}\cdots p_k^{\delta_k}$ for which $0\leq \delta_i \leq \gamma_i$ for all $1\leq i\leq k$. We have
		\[\begin{split}
		\bigcup_{r|r'}X_r & = \psi_n\left(\bigcup_{r|r'}Y_r\right)\\
		& = \psi_n\left(\bigcup_{\delta_1 = 0}^{\gamma_1} \cdots \bigcup_{\delta_k = 0}^{\gamma_k} S_1^{(\delta_1)}\times \cdots \times S_k^{(\delta_k)}\right)\\
		& = \psi_n\left(\prod_{i = 1}^n \left(\bigcup_{\delta_i = 0}^{\gamma_i} S_i^{(\delta_i)}\right)\right).
		\end{split}\]
		From our construction, we know that $\bigcup_{\delta_i = 0}^{\gamma_i} S_i^{(\delta_i)} = \emptyset$ if $p_i = 2$ and $\gamma_i < \alpha_i$, and otherwise
		\[\bigcup_{\delta_i = 0}^{\gamma_i} S_i^{(\delta_i)} = \{x: 0\leq x < p_i^{\gamma_i}\} \subseteq \Z_{p_i^{\alpha_i}}.\]
		Thus we may conclude that
		\begin{equation}\label{eqn:union-subset}\bigcup_{r|r'}X_r \subseteq \{\psi_n(t_1, \cdots, t_k): 0\leq t_i < p_i^{\gamma_i}\Forall1\leq i\leq k\}.\end{equation}
		Let $Z_{r'}$ be the set on the right hand side of \eqref{eqn:union-subset}. Any two distinct elements of $Z_{r'}$ differ on some $t_i$ for $1\leq i\leq k$, i.e., they have different remainders when divided by $p_i^{\gamma_i}$ for some $i$. Therefore their difference is not divisible by $r'$. Hence there is at most one element of $C(r', w)$ in $\bigcup_{r|r'}X_r \subseteq Z_{r'}$ for any $w\in \Zn$. We conclude that for any $r'$ that divides $n$ and any $w\in \Zn$,
		\[|\chi(C(r', w))| = \left|\chi\left(C(r', w) \cap \left(\bigcup_{r|r'}X_r\right)\right)\right| \leq 1.\]
		This finishes our proof.
	\end{proof}

		\begin{remark}
		By being a little more careful with bounds in the above proof, we may improve the latter inequality in Lemma \ref{lem:color-with-bounds-for-cong-classes} to
		\[\max_{A\in \CA_n}|\chi(A)| \leq \sqrt{n}\cdot 2^{\left(\frac{1}{2} + o(1)\right)k},\]
		where the $o(1)$ term tends to $0$ as $k \to \infty$. This in turn improves the bound on 
		$\disc(\CA_n)$ to $\min_{r|n}\left(\frac{n}{r} + \sqrt{r}\cdot 2^{\left(\frac{1}{2} + o(1)\right)\omega(r)}\right)$
		in Theorem~\ref{thm:bound-from-distribution}, where the $o(1)$ term goes to $0$ as $\omega(r) \to \infty$.
	\end{remark}


	\section{Lower Bounds }\label{sec lower bounds}
	In this section, we prove the lower bounds in Theorems~\ref{main theorem} and \ref{prime power}. These proofs use Fourier analysis. 
	We first set up helpful notations and prove a consequence of the Plancherel theorem in Lemma~\ref{lem:Plancherel}. 
	
	
	\begin{notations} 
		Let $f: \Zn \to \C$. Any coloring $\chi: \Zn \to \{1, -1\}$ is a special case. In this section, all summations are over $\Zn$ unless otherwise specified.
		
		We define the Fourier transform $\widehat f: \Zn \to \C$ by $\widehat f(r) = \sum_{x\in \Zn}f(x)e^{-\frac{2\pi i}{n}xr}$.
		
		For two functions $f_1, f_2: \Zn\to \C$, their convolution $f_1*f_2: \Zn \to \C$ is given by $f_1*f_2(a) = \sum_x f_1(x)f_2(a-x)$. We have the convolution identity $\widehat{f_1*f_2} = \widehat{f_1} \cdot \widehat f_2$.

		For any $r\in \N$ that divides $n$, we say that $x \equiv y \Mod r$ for $x, y\in \Zn$ if $x = y + kr$ for some $k\in \Zn$. We define, for any $r \in \N$ that divides $n$ and $a\in \Zn$,
		\[g_f(a, r) := \sum_{x\in \Zn: x\equiv a\!\Mod r} f(x) \quad \AND  \quad G_f(r) := \sum_{a=0}^{r-1} |g(a, r)|^2.\]
		We see that $g_f(a, r) = g_f(a+kr, r)$, so $g_f(\cdot, r)$ can be treated as a function on $\Z_r$.

		Let $M$ be a set or a multiset of elements in $\Zn$. For $a, r\in \Zn$, we define $rM$ to be the multiset $\{rx: x\in M\}$, and $a+M$ to be the multiset $\{a+x: x\in M\}$. We define $m_M$ to be the multiplicity function, i.e. $m_M(x)$ is the multiplicity of $x$ in $M$. In particular, when $M$ is a set, $m_M = 1_M$ is the indicator function. We also define 
		\[f(M) = \sum_{x\in M}f(x).\]
	\end{notations}
	
	
	\begin{lemma}\label{lem:Plancherel}
		For any $f: \Zn \to \C$, let $\widehat f$ be its Fourier transform. Let $r\in \N$ be a factor of $n$. Then
	\[\sum_{k = 0}^{r-1} \left|\widehat f\left(k\cdot \frac{n}{r}\right)\right|^2 = rG_f(r).\]
	\end{lemma}
	\begin{proof}
		Let $t = n/r$. For each $0\leq k\leq r-1$, we know that by definition
		\[\widehat f(kt) = \sum_{x = 0}^{n-1} f(x)e^{-\frac{2\pi i}{n}x\cdot kt} = \sum_{x = 0}^{tr-1} f(x)e^{-\frac{2\pi i}{r}xk} = \sum_{a = 0}^{r-1}g_f(a, r)e^{-\frac{2\pi i}{r}ak}.\]
		Hence for $g: \Z_r\to \C$ defined by $g(a) := g_f(a, r)$, its Fourier transform is $\widehat g(k) = \widehat f(kt)$. By the Plancherel theorem we have
		\[\sum_{k = 0}^{r-1} \left|\widehat f\left(k\cdot \frac{n}{r}\right)\right|^2 = \sum_{k = 0}^{r-1} |\widehat g(k)|^2 = r \sum_{a = 0}^{r-1} |g(a)|^2 = rG_f(r).\]
	\end{proof}
	
In Lemmas~\ref{lem:Fourier-RHS} and \ref{lem:LHS}, we prove lower and upper bounds on the same quantity respectively. Combining them we get Corollary~\ref{cor:lower-bound-general}. This corollary lower bounds the discrepancy of a function $f$ over arithmetic progressions in $\Zn$ by the Fourier coefficients of $f$. 
	
	The next lemma lower bounds the quantity by a weighted $L^{2}$ sum of Fourier coefficients $\widehat f$. The weights depend on the arithmetic structure of $n$. 
	\begin{lemma}\label{lem:Fourier-RHS}
		Let $f: \Zn \to \C$, $1\leq m\leq n$ be an integer, and $M:= \{0, 1 \dots, m-1\} \subseteq \Zn$. Then
		\begin{equation}\label{eqn:Fourier-RHS}\sum_{a\in \Zn}\sum_{b\in \Zn} |f(a+bM)|^2 \geq \sum_{r\in \Zn} |\widehat f(r)|^2 \cdot \max\left(\frac{m^2\cdot \gcd(r, n)}{n}, m\right).\end{equation}
	\end{lemma}
	\begin{proof}
	We show some properties about the multiplicity function. For the set $M$, we have
		\begin{equation}\label{eqn:convolution}
		f(a+bM) = \sum_{x\in M} f(a+bx) = \sum_x f(x)m_{a+bM}(x) = \sum_x f(x)m_{-bM}(a-x) = f * m_{-bM}(a).
		\end{equation}
		For any $u, v\in \Zn$, we have
		\begin{equation}\label{eqn:fourier-exchange-factor-in-multiplicity}
		\widehat {m_{-uM}}(v) = \sum_{x} m_{-uM}(x)e^{-\frac{2\pi i}{n}vx} = \sum_{x\in -uM} e^{-\frac{2\pi i}{n}vx} = \sum_{y\in -M} e^{-\frac{2\pi i}{n}vuy} =  \widehat{m_{-M}}(uv).
		\end{equation}
		Applying \eqref{eqn:fourier-exchange-factor-in-multiplicity} twice we have $\widehat {m_{-uM}}(v) = \widehat {m_{-vM}}(u).$ We have
		\begin{equation}\label{eqn:RHS}\begin{split}
		\sum_{a}\sum_{b} |f(a+bM)|^2 & = \sum_{b} \sum_a |f * m_{-bM}(a)|^2 = \sum_b \frac{1}{n} \sum_r \left|\widehat{f * m_{-bM}}(r)\right|^2\\
		& = \frac{1}{n}\sum_b \sum_r |\widehat f(r)|^2\cdot \left|\widehat {m_{-bM}}(r)\right|^2 = \frac{1}{n}\sum_b \sum_r |\widehat f(r)|^2 \left|\widehat {m_{-rM}}(b)\right|^2 \\
		& = \sum_r |\widehat f(r)|^2 \cdot \left( \frac{1}{n}\sum_b \left|\widehat {m_{-rM}}(b)\right|^2\right) = \sum_r |\widehat f(r)|^2 \cdot \sum_s |m_{-rM}(s)|^2,
		\end{split}
		\end{equation}where the first equality uses \eqref{eqn:convolution}, the second uses the Planchrel theorem, the third uses the Fourier identity of convolution, the fourth uses \eqref{eqn:fourier-exchange-factor-in-multiplicity} twice, and the last uses the Plancherel theorem.
		Now we evaluate $\sum_s |m_{-rM}(s)|^2$. Let $k = \gcd(r, n)$. When $km \leq n$, we know that $rM$ contains $m$ elements with multiplicity 1. Thus in this case we have
		\[\sum_s |m_{-rM}(s)|^2 = m.\]
		
		Otherwise when $km > n$, we know that $rM$ contains elements that are multiples of $k$. There are $\frac nk$ such elements in $\Zn$. Therefore, by the Cauchy-Schwarz inequality, we have
		\[\sum_s |m_{-rM}(s)|^2 = \sum_{t = 0}^{\frac{n}{k} - 1} |m_{-rM}(tk)|^2 \geq \frac{\left(\sum_{t = 0}^{\frac{n}{k} - 1} |m_{-rM}(tk)|\right)^2}{n/k} = \frac{m^2k}{n}.\]
		Hence we may conclude that for any $r \in \Zn$,
		\begin{equation}\label{eqn:CS} \sum_s |m_{-rM}(s)|^2 \geq \max\left(\frac{m^2\cdot \gcd(r, n)}{n}, m\right).\end{equation}
		Combining this with \eqref{eqn:RHS}, we obtain \eqref{eqn:Fourier-RHS}.
	\end{proof}
	
	The next lemma gives an upper bound on the quantity above. The upper bound involves the discrepancy of $f$ over $\CA_n$ and some arithmetic sums of $f$. Let $\phi(\cdot)$ be Euler's totient function.
	
	\begin{lemma}\label{lem:LHS}
		Let $f: \Zn \to \C$. Suppose that 
		\[T_f := \max_{A\in \CA_n} |f(A)|.\]
		Let $m\leq n$ be a positive integer, and $M:= \{0, 1, \dots, m-1\} \subseteq \Zn$. Then we have
		\begin{equation}\label{eqn:LHS}
		\sum_{a\in \Zn}\sum_{b\in \Zn} |f(a+bM)|^2 \leq n^2 T_f^2 + \sum_{0\leq k < m: k|N} m^2\frac{\phi(k)}{k}G_f(n/k).
		\end{equation}
	\end{lemma}
	\begin{proof}
		For each fixed $b\in \Zn$, we analyze the summation $\sum_a |f(a+bM)|^2$.
		Let $r = \gcd(b, n)$ and $k = n/r$. There are two possibilities: $rm \leq n$ or $rm > n$.
		
		If $rm \leq n$, then each element in $bM$ has multiplicity $1$. Therefore $a + bM$ is a set for all $a\in \Zn$. Moreover from our definition of $M = \{0, 1, \dots, m-1\}$, we know that $a+bM\in \CA_n$ is an arithmetic progression in $\Zn$. In this case, for any $a$, $|f(a+bM)| \leq T_f$, so
		\begin{equation}\label{eqn:LHS-case-1}
		\sum_a |f(a+bM)|^2 \leq \sum_a T_f^2 = nT_f^2.
		\end{equation}
		
		If $rm > n$, i.e., $m > k$, then some elements in $bM$ have multiplicity greater than $1$. The bound in \eqref{eqn:LHS-case-1} no longer applies. Let $K = \{0, 1, \dots, k-1\}$ and it follows that each element in $bK$ has multiplicity $1$. Moreover, it covers each multiple of $r$ in $\Zn$ exactly once. Hence
		\[f(a+bK) = \sum_{t= 0}^{k-1} f(a+bt) = g_f(a, r).\]
		Suppose that $m = qk+s$ for integers $q$ and $0\leq s < k$. Then we may write $m_M = m_{S} + \sum_{j = 0}^{q-1} m_{s+jk+K}$ where $S = \{0, 1, \dots, s-1\}$.
		Thus, for any $a\in \Zn$, we have
		\begin{equation}\label{eqn:lem-4.4-1}\begin{split}f(a + bM) & = f(a + bS) + \sum_{j = 0}^{q-1}f(a+b\cdot (s+jk) + bK) \\
		& = f(a+bS) + \sum_{j = 0}^{q-1}g_f(a+b\cdot (s+jk), r) = f(a+bS) + q\cdot g_f(a, r).
		\end{split}\end{equation}
		In the last step, we use the fact that $b\cdot (s+jk)$ is a multiple of $r$.
		
		Since $s < k$, each element in $bS$ has multiplicity at most $1$, and further $a+bS\in \CA_n$ for all $a$. Thus 
		\begin{equation}\label{eqn:lem-4.4-2}
		    |f(a+bS)| \leq T_f.
		\end{equation}
		We now partition $a\in \Zn$ into congruence classes mod $r$. In particular, we have
		\begin{equation}\label{eqn:lem-4.4-3}
		    \sum_{a}|f(a+bM)|^2 = \sum_{i = 0}^{r-1} \sum_{a\equiv i \!\Mod r} |f(a+bM)|^2.
		\end{equation}
		For each $a'\equiv i\Mod r$, there are exactly $s$ choices of $a\equiv i\Mod r$ such that $a'\in a+bS$. Thus,
		\begin{equation}\label{eqn:lem-4.4-4}
		 \sum_{j = 0}^{k-1} f(i+jr+bS) =\sum_{a\equiv i\!\Mod r} f(a+bS) =  s\cdot g_f(i, r).
		\end{equation}
		Consequently, we have
		\[\begin{split}
		\sum_{a\equiv i\!\Mod r} |f(a+bM)|^2 & = \sum_{j = 0}^{k-1} |f(i+jr+bM)|^2 \\
		[\textrm{by }\eqref{eqn:lem-4.4-1}]\quad & = \sum_{j = 0}^{k-1} |f(i+jr+bS) + q\cdot g_f(i, r)|^2 \\
		 & = \sum_{j = 0}^{k-1} \left(|f(i+jr+bS)|^2 + 2 \Re\left(q\cdot f(i+jr+bS)\overline{g_f(i, r)}\right) + q^2|g_f(i, r)|^2\right)\\
		[\textrm{by }\eqref{eqn:lem-4.4-2}]\quad & \leq kT_f^2 + kq^2|g_f(i, r)|^2 + 2q\Re\left(\overline{g_f(i, r)}\sum_{j = 0}^{k-1} f(i+jr+bS)\right) \\
		[\textrm{by }\eqref{eqn:lem-4.4-4}]\quad & = kT_f^2 + kq^2|g_f(i, r)|^2 + 2qs|g_f(i, r)|^2 \\
	    & \leq kT_f^2 + |g_f(i, r)|^2\cdot \frac{k^2q^2 + 2kqs + s^2}{k} \\
		[\textrm{by }m=kq+s]\quad & = kT_f^2 + \frac{m^2}{k}|g_f(i, r)|^2.
		\end{split}\]
		Put this into \eqref{eqn:lem-4.4-3}. We know that in the second case where $rm > n$, i.e., $k < m$,
		\begin{equation}\label{eqn:LHS-case-2}
		\sum_{a}|f(a+bM)|^2 =\sum_{i = 0}^{r-1} \sum_{a\equiv i \!\Mod r} |f(a+bM)|^2 \leq nT_f^2 + \frac{m^2}{k}G_f(r) =  nT_f^2 + \frac{m^2}{k}G_f(n/k).\end{equation}
		The second case happens exactly when $\gcd(b, n) = n/k$ for some $k < m$ which divides $n$. The number of such choices of $b$ is exactly $\phi(k)$ for any fixed $k$. Hence combining \eqref{eqn:LHS-case-1} and \eqref{eqn:LHS-case-2}, we get \eqref{eqn:LHS}.
	\end{proof}
	Combining Lemma~\ref{lem:Fourier-RHS} and Lemma~\ref{lem:LHS}, we deduce the following general lower bounds.  
	\begin{corollary}\label{cor:lower-bound-general}
		Let $f: \Zn \to \C$ and $\widehat f$ be its Fourier transform. Suppose that 
		\[T_f := \max_{A\in \CA_n} |f(A)|.\]
		For any positive integer $m \leq n$, we have
		\begin{equation}\label{eqn:cor-lower-bound-general}
		n^2T_f^2 +  \sum_{1\leq k < m: k|n} \frac{m^2\phi(k)}{k}G_f(n/k) \geq \sum_{r\in \Zn} |\widehat f(r)|^2 \cdot \max\left(\frac{m^2\cdot \gcd(r, n)}{n}, m\right).
		\end{equation}
	\end{corollary}
	
	In Corollary~\ref{cor:lower-bound-general}, the inequality involves $G_f$ and $\widehat f$ besides $T_f$.  We aim to get a lower bound on $\disc(\CA_n)$ that only depends on arithmetic structure of $n$ (Proposition~\ref{prop:lower-bound-unsimplified}). To achieve that, we need two lemmas (Lemma~\ref{lem:eqn-from-mobius-transform} and Lemma~\ref{lem:ineq-from-mobius}) to remove the dependency on $\widehat f$ and $G_f$.
	\begin{lemma}\label{lem:eqn-from-mobius-transform}
		Let $f: \Zn \to \C$ and $\widehat f$ be its Fourier transform and $m\leq n$ be a positive integer. Then
		\[ \sum_{1\leq k \leq n: k|n} \frac{m^2\phi(k)}{k}G_f(n/k) = \sum_{r\in \Zn} |\widehat f(r)|^2 \cdot \frac{m^2\cdot \gcd(r, n)}{n}.\]
	\end{lemma}
	\begin{proof}
		By Lemma~\ref{lem:Plancherel}, we know that
		\[\frac{n}{k}G_f(n/k) = \sum_{r\in \Zn: k|r}|\widehat f(r)|^2.\]
		Thus we have (noticing that $\sum_{k|l}\phi(k) = l$ for any positive integer $l$)
		\[\begin{split}
		\sum_{1\leq k \leq n: k|n} \frac{m^2\phi(k)}{k}G_f(n/k) & = \sum_{1\leq k \leq n: k|n} \frac{m^2\phi(k)}{n}\sum_{r\in \Zn: k|r}|\widehat f(r)|^2  = \sum_{r\in \Zn}|\widehat f(r)|^2\cdot \frac{m^2}{n} \sum_{1\leq k\leq n: k|n, k|r}\phi(k)\\
		& = \sum_{r\in \Zn}|\widehat f(r)|^2\cdot \frac{m^2}{n} \sum_{1\leq k\leq n: k|\gcd(n, r)}\phi(k) =  \sum_{r\in \Zn}|\widehat f(r)|^2\cdot \frac{m^2}{n}\gcd(r,n).
		\end{split}\]
		Hence we have the desired equality.
	\end{proof}
	
	\begin{lemma}\label{lem:ineq-from-mobius}
		Let $f: \Zn \to \C$ and $\widehat f$ be its Fourier transform. Let $m,l \leq n$ be positive integers. Then
		\begin{equation}\label{eqn:ineq-from-mobius}\sum_{r\in \Zn} |\widehat f(r)|^2 \cdot \min\left(\frac{m^2\cdot \gcd(r, n)}{n}, m\right) \leq \sum_{1\leq k \leq l: k|n}\frac{m^2\phi(k)}{k}G_f(n/k) + \sum_{l < k \leq n: k|n}\frac{mn}{k}G_f(n/k).\end{equation}
	\end{lemma}
	\begin{proof}
		It suffices to compare the coefficient of each $|\widehat f(r)|^2$ on both sides after expanding all terms of $G_f$ on the right hand side \eqref{eqn:ineq-from-mobius} using Lemma~\ref{lem:Plancherel}. By Lemma~\ref{lem:Plancherel}, for any $k$ that divides $n$,
		\[\frac{n}{k}G_f(n/k) = \sum_{r\in \Zn: k|r}|\widehat f(r)|^2.\]
		
		For each $r\in \Zn$, the coefficient $t_r$ of $|\widehat f(r)|^2$ on the right hand side of \eqref{eqn:ineq-from-mobius} is given by
		\[\begin{split}
		t_r & = \sum_{1\leq k \leq l: k|\gcd(r, n)}\frac{m^2\phi(k)}{n} + \sum_{l < k \leq n: k|\gcd(r, n)}m  = \frac{m^2}{n}\sum_{1\leq k\leq l: k|\gcd(r, n)} \phi(k)+ m\sum_{l < k \leq n: k|\gcd(r, n)}1.
		\end{split} \]
		If $\gcd(r, n) \leq l$, then we know that the first summation sums over all factors of $\gcd(r, n)$, and the second summation is zero. Hence we know that in this case,
		\[t_r = \frac{m^2}{n}\gcd(r, n).\]
		If $\gcd(r, n) > l$, then in particular the second summation contains at least one term $k = \gcd(r, n)$, and the first summation is nonnegative. Hence in this case,
		\[t_r > m.\]
		We may conclude that $t_r \geq \min\left(\frac{m^2\cdot \gcd(r, n)}{n}, m\right)$ for all $r\in \Zn$. This yields \eqref{eqn:ineq-from-mobius}.
	\end{proof}

	Using the lemmas above, we prove a lower bound on $\disc(\CA_n)$ which depends only on the arithmetic structure of $n$.
	\begin{proposition}\label{prop:lower-bound-unsimplified}
		For any positive integers $n$ and $l\leq n$,
		\[\frac{1}{(\disc(\CA_n))^2} \leq \frac{8}{n}\sum_{1\leq k\leq l: k|n} \phi(k) + 2\sum_{l < k \leq n: k|n}\frac{1}{k^2}.\]
	\end{proposition}
	\begin{proof}
		For simplicity let us denote
		\[S_1 := \sum_{1\leq k\leq l: k|n} \phi(k) \quad \AND  \quad S_2 := \sum_{l < k \leq n: k|n}\frac{1}{k^2}.\]
		We aim to show 
		\begin{equation}\label{eqn:prop-lower-bound-unsimplified-target} \disc(\CA_n)^2 \geq \frac{1}{\frac{8}{n}S_1 + 2S_2}. \end{equation}
		Because we always have $\disc(\CA_n) \geq 1$, we may assume that $S_1 < \frac{n}{8}$.
		
		Let $\chi: \Zn \to \{1, -1\}$ be any coloring of $\Zn$, and let $T_\chi := \max_{A\in \CA_n} |\chi(A)|.$ Let $m \leq n$ be a positive integer to be determined later. By Corollary~\ref{cor:lower-bound-general}, we know that
		\begin{equation}\label{eqn:divisor-general}
		n^2T_\chi^2 +  \sum_{1\leq k < m: k|n} \frac{m^2\phi(k)}{k}G_\chi(n/k) \geq \sum_{r\in \Zn} |\widehat \chi(r)|^2 \cdot \max\left(\frac{m^2\cdot \gcd(r, n)}{n}, m\right).\end{equation}
		By Lemma~\ref{lem:eqn-from-mobius-transform}, we know that
		\[ \sum_{1\leq k \leq n: k|n} \frac{m^2\phi(k)}{k}G_\chi(n/k) = \sum_{r\in \Zn} |\widehat \chi(r)|^2 \cdot \frac{m^2\cdot \gcd(r, n)}{n}.\]
		Subtract both sides from \eqref{eqn:divisor-general}. We get
		\begin{equation}\label{eqn:divisor-main}
		n^2T_\chi^2 + \sum_{r\in \Zn} |\widehat \chi(r)|^2 \cdot \min\left(\frac{m^2\cdot \gcd(r, n)}{n}, m\right) \geq \sum_{r\in \Zn} |\widehat \chi(r)|^2\cdot m + \sum_{m\leq k \leq n: k|n} \frac{m^2\phi(k)}{k}G_\chi(n/k).
		\end{equation}
		Since $\chi$ takes value in $\{-1, 1\}$, we know that
		\begin{equation}\label{eqn:G-const}\sum_{r\in \Zn}|\widehat\chi(r)|^2 = nG_\chi(n) = n^2.\end{equation}
		Note that each $G_\chi(\cdot)$ is nonnegative. By \eqref{eqn:G-const}, the right hand side of \eqref{eqn:divisor-main} is lower bounded by
		\begin{equation}\label{eqn:divisor-main-RHS}
		\sum_{r\in \Zn} |\widehat \chi(r)|^2\cdot m + \sum_{m\leq k \leq n: k|n} \frac{m^2\phi(k)}{k}G_\chi(n/k) \geq n^2m.
		\end{equation}
		For the left hand side of \eqref{eqn:divisor-main}, we apply Lemma~\ref{lem:ineq-from-mobius} for $m$ and $l$ and get
		\begin{equation}\label{eqn:divisor-main-LHS-1}
		\sum_{r\in \Zn} |\widehat \chi(r)|^2 \cdot \min\left(\frac{m^2\cdot \gcd(r, n)}{n}, m\right) \leq \sum_{1\leq k\leq l:k|n}\frac{m^2\phi(k)}{k}G_{\chi}(n/k) + \sum_{l < k \leq n: k|n} \frac{mn}{k}G_\chi(n/k). 
		\end{equation}
		
		Note that by definition, each single $g_\chi(a, n/k)$ is the sum of $k$ values of $\chi$. Because $\chi$ takes value in $\{-1, 1\}$, we have
		\begin{equation}\label{eqn:G-bound-as-const}
		G_\chi(n/k) \leq \frac{n}{k} \cdot k^2 = nk.
		\end{equation}
		Moreover, each $g_\chi(a, n/k)$ is $\chi(A)$ for some $A\in \CA_n$, so $|g_\chi(a, n/k)| \leq T_\chi$. This means that for all $k$,
		\begin{equation}\label{eqn:G-bound-as-T}
		G_\chi(n/k) \leq \frac{n}{k}\cdot T_\chi^2.
		\end{equation}
		We bound the first term on the right hand side of \eqref{eqn:divisor-main-LHS-1} using \eqref{eqn:G-bound-as-const} and get
		\begin{equation}\label{eqn:divisor-main-LHS-2}
		\sum_{1\leq k\leq l:k|n}\frac{m^2\phi(k)}{k}G_{\chi}(n/k) \leq m^2n \sum_{1\leq k\leq l:k|n} \phi(k) = nm^2\cdot S_1.
		\end{equation}
		On the other hand, we bound the second term on the right hand side of \eqref{eqn:divisor-main-LHS-1} using \eqref{eqn:G-bound-as-T} and get
		\begin{equation}\label{eqn:divisor-main-LHS-3}
		\sum_{l < k \leq n: k|n} \frac{mn}{k}G_\chi(n/k) \leq  \sum_{l <  k \leq n: k|n} \frac{mn^2}{k^2}T_\chi^2 = n^2mT_\chi^2\cdot S_2.
		\end{equation}
		
		Put \eqref{eqn:divisor-main-LHS-2} and \eqref{eqn:divisor-main-LHS-3} into \eqref{eqn:divisor-main-LHS-1}. We get
		\begin{equation}\label{eqn:divisor-main-LHS}
		\sum_{r\in \Zn} |\widehat \chi(r)|^2 \cdot \min\left(\frac{m^2\cdot \gcd(r, n)}{n}, m\right) \leq nm^2 \cdot S_1 + T_\chi^2n^2mS_2.
		\end{equation}
		
		Finally we put \eqref{eqn:divisor-main-RHS} and \eqref{eqn:divisor-main-LHS} into \eqref{eqn:divisor-main}. We get
		\[\left(n^2 + n^2mS_2\right)T_\chi^2 \geq n^2m -nm^2 S_1.\]
		Dividing both sides by $n^2m$, we get
		\begin{equation}\label{eqn:divisor-conclusion}
		\left(\frac{1}{m} + S_2\right)T_\chi^2\geq 1 - \frac{m}{n}S_1.
		\end{equation}
		Now we pick $m = \lfloor n/(2S_1)\rfloor$. Note that $S_1 < n/8$ from our assumption, so $m \geq \frac{n}{4S_1}$. Therefore,
		\[\frac{1}{m} \leq \frac{4}{n}S_1 \quad \AND \quad 1- \frac{m}{n}S_1 \geq \frac{1}{2}.\]
		Put them into \eqref{eqn:divisor-conclusion}. We conclude that
		\[ T_\chi^2  \geq \frac{1 - \frac{m}{n}S_1}{\frac{1}{m} + S_2} \geq \frac{1/2}{\frac{4}{n}S_1 + S_2} = \frac{1}{\frac{8}{n}S_1 + 2S_2}.\]
		Since this bound applies to all $T_\chi^2$, it also applies to $(\disc(\CA_n))^2 = \min_{\chi}T_\chi^2$. Hence we have \eqref{eqn:prop-lower-bound-unsimplified-target}.
	\end{proof}
	\begin{remark}\label{rem:unit 2}As mentioned earlier in Remark~\ref{rem: unit 1},
		the proof above also applies to the case where $\chi$ takes value in the unit circle on the complex plane $\{z\in \C: |z| = 1\}$ (instead of $\{1, -1\}$). Just note that \eqref{eqn:G-const} and \eqref{eqn:G-bound-as-const} hold in this more general case as well, and all other steps are identical. 
	\end{remark}
	
	Now we prove two corollaries. The first shows that the upper bound in Theorem~\ref{thm:bound-from-distribution} is tight up to an $n^{o(1)}$ factor.
	\begin{corollary}[Lower bound in Theorem~\ref{main theorem}]\label{cor:lower-bound-with-divisor}
		There exists an absolute constant $c > 0$ such that, for any positive integer $n$,
		\[\disc(\CA_n) \geq \frac{1}{8\sqrt{d(n)}} \cdot \min_{r|n}\left(\frac{n}{r} + \sqrt{r}\right),\]
		where $d(n)$ is the number of factors of $n$.
	\end{corollary}
	\begin{proof}Let $t_1$ be the minimum factor of $n$ that is at least $n^{\frac23}$, and let $t_2$ be the maximum factor of $n$ less than $n^{\frac23}$. As     \[2\cdot \min\left(\sqrt{t_1}, \frac{n}{t_2}\right) \geq \min\left(\sqrt{t_1}+\frac{n}{t_1}, \sqrt{t_2}+\frac{n}{t_2}\right) = \min_{r|n}\left(\frac{n}{r} + \sqrt{r}\right),\]
	it suffices to show that 

		\begin{equation}\label{eqn:simplified}\disc(\CA_n) \geq \frac{1}{4\sqrt{d(n)}}\cdot \min\left(\sqrt{t_1}, \frac{n}{t_2}\right).\end{equation}
		
		Now we apply Proposition~\ref{prop:lower-bound-unsimplified} to $n$ and $l = \frac{n}{t_1}$. We bound the two summations as follows. We have
		\begin{equation}\label{eqn:divisor-main-term-1}
		\sum_{1\leq k\leq \frac{n}{t_1}:k|n} \phi(k) \leq \sum_{1\leq k\leq \frac{n}{t_1}:k|n} k \leq \frac{n}{t_1}\cdot d(n).
		\end{equation}
		Note that $t_2$ is the largest factor of $n$ less than $t_1$, so the minimum $k$ in $\{\frac{n}{t_1} < k \leq n: k|n\}$ is $n/t_2$. Thus we have
		\begin{equation}\label{eqn:divisor-main-term-2}
		\sum_{\frac{n}{t_1} < k \leq n: k|n} \frac{1}{k^2} \leq  \sum_{\frac{n}{t_2} \leq k \leq n: k|n} \frac{1}{k^2} \leq \frac{t_2^2}{n^2}\cdot d(n).
		\end{equation}
		
		Using above two bounds, we get
		\[(\disc(\CA_n))^2 \geq \frac{1}{\frac{8}{t_1}d(n) + 2\frac{t_2^2}{n^2}d(n)} \geq \frac{1}{16d(n)}\min\left(t_1, \frac{n^2}{t_2^2}\right).\]
		This is equivalent to \eqref{eqn:simplified}, so we have the expected inequality.
	\end{proof}
	
	The second corollary proves the lower bound in Theorem~\ref{prime power}. The main observation is that the factor $d(n)$  can be removed in \eqref{eqn:divisor-main-term-1} and \eqref{eqn:divisor-main-term-2} when $n$ is a prime power.
	\begin{corollary}[Lower bound in Theorem~\ref{prime power}]\label{cor:lower bound prime powers}Let $p$ be a prime number, and $k$ be a positive integer. Then for $n = p^k$,
		\begin{equation}\label{eqn:lower bound prime powers}\disc(\CA_n) \geq \frac{1}{4}p^{\frac{k - \lfloor k/3\rfloor}{2}}.\end{equation}
	\end{corollary}
	\begin{proof}
		In Proposition~\ref{prop:lower-bound-unsimplified}, we pick $l = p^t$ for some $1\leq t < n$ to be determined later. 
		Note that the factors of $n$ between $1$ and $l$ are given by $s = p^i$ for $0\leq i\leq t$, which are all factors of $l$. Hence we have
		\[\sum_{1\leq s\leq l: s|n}\phi(s) = l = p^t.\]
		All factors of $n$ larger than $l$ are given by $s = p^{i}$ for $t+1\leq i \leq k$. Thus
		\[\sum_{l < s \leq l: s|n}\frac{1}{s^2} = \sum_{i = t+1}^k p^{-2i} \leq p^{-2t-2}\cdot \left(1 + \frac{1}{p^2} + \frac{1}{p^{4}}+ \cdots\right) \leq 2p^{-2t-2}.\]
		
		Therefore, by Proposition~\ref{prop:lower-bound-unsimplified} we get
		\[\disc(\CA_n) \geq \sqrt{\frac{1}{8p^{t-k} + 4p^{-2t-2}}} \geq \frac{1}{4}\cdot \min\left(p^{\frac{k-t}{2}}, p^{t+1}\right) = \frac{1}{4}\cdot p^{\min(\frac{k-t}{2}, t+1)}.\]
		
		We pick $t = \lfloor{k/3}\rfloor$ to get \eqref{eqn:lower bound prime powers}.
		
	\end{proof}
	
	\section{Concluding Remarks}\label{concludingremarks}
	In Theorem~\ref{main theorem}, the upper and lower bounds are off by a factor of $O(d(n)^{3/2})$. We were not able to close this gap, and it seems that major improvement of either bound would require new observations.
	\begin{problem}Determine $\disc(\CA_n)$ up to a constant factor for all $n$.
	\end{problem}
	
	There are other notions of discrepancy besides the one studied in this paper. Among them there is the {\it hereditary discrepancy}, defined for a system $(\Omega, \CA)$ as
	\[\herdisc(\CA) := \max_{X\subseteq \Omega} \disc(\CA|_X)\]
	where $\CA|_X = \{A\cap X: A\in \CA\}$. Clearly $\disc(\CA) \leq \herdisc(\CA)$. For the set $\CA$ of arithmetic progressions in $[n]$, Matou\v{s}ek and Spencer \cite{MS} proved a stronger statement that $\herdisc(\CA) = O(n^\frac14)$. This is because their partial coloring method works not just for $\CA$, but also for $\CA|_X$ for any $X\subseteq [n]$.
	
	In contrast, our construction of the coloring in Section~\ref{section upper} is only valid for coloring the whole set $\Zn$. While it can be adapted so that the same upper bound (possibly with a larger constant factor) applies to some special subsets of $\Zn$, it does not work for all subsets $X\subseteq \Zn$.
	\begin{problem}
		Estimate the hereditary discrepancy of $\CA_n$.
	\end{problem}
	By Corollary~\ref{cor:subset-coloring} we have $\herdisc(\CA_n) = O(n^\frac12)$. The method used in Section~\ref{entropy} can be adapted to give the following slightly stronger statement. Let $\phi(\cdot)$ be Euler's totient function.
	\begin{theorem}~\label{hereditary}
		There exists a constant $c$ such that 
		for all positive integers $n$, we have 
		\[\herdisc(\CA_n) \le c \phi(n)^{\frac{1}{2}}\left(\log \frac{en}{\phi(n)}\right)^{\frac{3}{2}}.  \]
	\end{theorem}
	We leave the proof of this theorem to the Appendix. It shows that the upper bound $O(n^{\frac12})$ is not always tight. It would be interesting to determine if there is a matching lower bound of the form $n^{\frac12 - o(1)}$.

\appendix	
	\section{Improved bound for hereditary discrepancy} 
	
	In this appendix, we prove Theorem \ref{hereditary}, an improved upper bound on the hereditary discrepancy of modular arithmetic progressions. We need the following lemma 
	which gives a partial coloring bound. 
	\begin{lemma}\label{lemforhered}
	Let $X\subseteq \Zn$ be a set of size $m > 0$. There exists an absolute constant $c$ such that there is a partial coloring $\chi: X\to \{-1, 0, 1\}$ that assigns $\pm 1$ to at least $m/10$ elements in $X$ such that
	\[\max_{A\in \mathcal{A}_n}|\chi(A\cap X)| \leq c\phi(n)^{\frac12}\left(\log\frac{en}{\phi(n)}\right)^{\frac12}.\]
	\end{lemma}
	
	Assuming this lemma, we may now prove Theorem \ref{hereditary}. 
	\begin{proof}[Proof of Theorem \ref{hereditary} assuming \ref{lemforhered}]
	Let $X\subseteq \Zn$ be a set of size $m$. We show that there exists a coloring $\chi: X\to \{1, -1\}$ such that, for some absolute constant $c$,
	\begin{equation}\label{eqn:hereditary-goal}
	    \max_{A\in \CA_n}|\chi(A\cap X)| \leq c \phi(n)^{\frac{1}{2}}\left(\log \frac{en}{\phi(n)}\right)^{\frac{3}{2}}.
	\end{equation}
	
		The idea is that we iteratively apply Lemma~\ref{lemforhered} to the set of uncolored elements until there are at most $\phi(n)$ elements left, and then apply Corollary~\ref{cor:subset-coloring} to color the remaining elements. Let $c_0, c_1$ be the constants in Lemma~\ref{lemforhered} and Corollary~\ref{cor:subset-coloring}.
		
		Start with $X_0 = X$. For each $i\geq 0$, we apply Lemma~\ref{lemforhered} to $X_i$ to get a partial coloring $\chi_i: X_i \to \{-1, 0, 1\}$, and we let $X_{i+1} = \chi_i^{-1}(0) \subseteq X_i$ to be the set of uncolored elements in $i$-th iteration. We continue this process until the $(k-1)$-th iteration where there are at most $\phi(n)$ elements left (i.e. $|X_{k}| \leq \phi(n)$). Then we apply Corollary~\ref{cor:subset-coloring} to get a coloring $\chi_{k}: X_{k}\to \{-1, 1\}$. Let $\chi$ be the final coloring given by $\chi(x_i) = \chi_i(x_i)$ if $x_i\in X_i\setminus X_{i+1}$ for $0\leq i\leq k-1$, and $\chi(x_{k}) = \chi_{k}(x_{k})$. We know that $|X_i| \leq 0.9|X_{i-1}|$ for all $1\leq i\leq k$, so $|X_k| \leq (0.9)^k|X_0| = (0.9)^km \leq (0.9)^kn$.
		Because we stop when there are at most $\phi(n)$ elements left, we shall see that $k \leq 1 + \log_{0.9}\frac{m}{n} \leq 10\log\frac{en}{m}$.
		
		Applying the bounds on the discrepancy of $\chi_i$ from Lemma~\ref{lemforhered} and Corollary~\ref{cor:subset-coloring}, we conclude that for any $A\in \CA_n$,
		\[
		\begin{split}
		|\chi(A\cap X)| & = \left|\sum_{i=0}^k \chi_i(X_i\cap A)\right| \leq \sum_{i=0}^{k-1} \left|\chi_i(X_i\cap A)\right| + \left|\chi_k(X_k\cap A)\right| \\
		& \leq k\cdot c_0\phi(n)^{\frac12}\left(\log\frac{en}{\phi(n)}\right)^{\frac12} + c_1|X_k|^{\frac12}\left(\log\frac{en}{|X_k|}\right)^{\frac12}\\
		& < 10c_0\phi(n)^{\frac12}\left(\log\frac{en}{m}\right)^\frac 32 + c_1\phi(n)^{\frac12}\left(\log\frac{en}{m}\right)^\frac 12 \leq c\phi(n)^{\frac12}\left(\log\frac{en}{m}\right)^\frac 32
		\end{split}\]
		for $c = 10c_0+c_1$ being an absolute constant. Since it holds for all $A\in \CA_n$, we have \eqref{eqn:hereditary-goal}.
	\end{proof}
	We are left to prove Lemma \ref{lemforhered}. We first need a slight generalization of the partial coloring lemma in Lemma~\ref{lem:partial-coloring-lemma}. In Lemma~\ref{lem:partial-coloring-lemma} we require $\Delta_S \geq 2\sqrt{|S|}$. Here we allow $\Delta_S$ to take any positive value.
	\begin{lemma}[Section 4.6 in \cite{Mat}]\label{lem:partial-coloring-lemma-full}
	Let $(V, C)$ be a set system on $n$ elements, and let a number $\Delta_S > 0$ be given for each set $S\in C$. Suppose that
	\begin{equation}
	    \sum_{S\in C:S\ne \emptyset}g\left(\frac{\Delta_S}{\sqrt{|S|}}\right) \leq \frac{n}{5}
	\end{equation}
	where
	\begin{equation}\label{eqn:def-g}
	    g(\lambda) = \begin{cases} 10e^{-\lambda^2/4} & \textrm{if }\lambda \geq 2, \\ 10\log(1 + 2\lambda^{-1}) & \textrm{if } 0 < \lambda < 2. \end{cases}
	\end{equation}
	Then there exists a partial coloring $\chi$ that assigns $\pm1$ to at least $n/10$ variables (and $0$ to the rest), satisfying $|\chi(S)| \leq \Delta_S$ for each $S\in C$.
	\end{lemma}
	
	We now prove Lemma \ref{lemforhered}. 
	
	\begin{proof}[Proof of Lemma \ref{lemforhered}]
	We use the same decomposition as in Lemma~\ref{lem:partial-coloring} to get the family $C_2$ of nonempty subsets of $X$. From the argument in Lemma~\ref{lem:partial-coloring}, we know that sets in $C_2$ are of size $2^i$ for $2^i \leq m$. If $\chi: X\to \{-1, 0, 1\}$ is such that for all $0\leq i\leq \log_2 m$ we have $\chi(S) \leq \Delta_i$ for any set $S\in C_2$ of size $2^i$, then 
	\begin{equation}\label{eqn:decomp-A}
	    \max_{A\in \CA_n}|\chi(A\cap X)| \leq 4\sum_{0\leq i\leq \log_2m}\Delta_i.
	\end{equation}
	Then we apply a better bound on the number of sets in $C_2$ of size $2^i$, which we shall denote as $f_i$. The notations here are the same as in Lemma~\ref{lem:partial-coloring}. For each $1\leq d < n$ and $0\leq a < \gcd(d, n)$, the number of choices of $t$ in \eqref{eqn:define-C2} is $\lfloor l_{d, a}/2^i\rfloor$. Yet note that $l_{d, a} \leq \frac{n}{\gcd(d, n)}$. Hence if $\gcd(d, n) > \frac{n}{2^i}$, then there are no such sets included in the set $X_{d, a}$. Therefore we have the following better bound on $f_i$:
	\begin{equation*}
	    f_i \leq \sum_{1\leq d \leq n-1: \gcd(d, n) \leq \frac{n}{2^i}}\sum_{a=0}^{\gcd(d, n) - 1}\left\lfloor\frac{l_{d, a}}{2^i}\right\rfloor = \sum_{1\leq d \leq n-1: \gcd(d, n) \leq \frac{n}{2^i}}\frac{m}{2^i} \leq \sum_{1\leq l \leq \frac{n}{2^i}:l|n}\phi(n/l)\cdot \frac{m}{2^i}.
	\end{equation*}
	Because $\phi(ab) \geq \phi(a)\phi(b)$ for all $a, b\in \N$, we know that $\phi(n/l) \leq \phi(n)/\phi(l)$ for all $l$ that divides $n$. Also note that $0 < \phi(n)/\phi(l)$ for all $l$ that does not divide $n$. By Landau \cite[p. 184]{Landau}, for any $x\geq 1$, $\sum_{l \leq x}1/\phi(l) \leq c_0\log ex$ for some absolute constant $c_0$. Hence we have
    \begin{equation}\label{eqn:bound-num-of-sets-A}f_i \leq \frac{m}{2^i}\sum_{1\leq l\leq \frac{n}{2^i}: l|n}\phi(n/l) \leq \frac{m}{2^i}\phi(n)\sum_{1\leq l\leq \frac{n}{2^i}}\frac{1}{\phi(l)} \leq c_0\frac{m}{2^i}\phi(n)\log\frac{en}{2^i}.
    \end{equation}
    
    For simplicity we denote $M = \phi(n)\log\frac{en}{\phi(n)}$. We define $b: (0, n] \to \R$, given by
    \begin{equation}\label{eqn:b-A}
        b(s) = \begin{cases} c_1\sqrt{s}\cdot \left(\frac{s}{M}\right)^{-1} & \textrm{if } s \geq M, \\ c_1\sqrt{s}\cdot \left(\frac{s}{M}\right)^{-0.1} & \textrm{if } s < M, \end{cases}
    \end{equation}
    where $c_1 > 2$ is an absolute constant to be determined later. We would like to show that there exists a partial coloring that colors at least $m/10$ elements in $X$, such that for any $S\in C_2$,
    \begin{equation}\label{eqn:lem-outcome-A}
        |\chi(S)| \leq b(|S|).
    \end{equation}
    In order to apply Lemma~\ref{lem:partial-coloring-lemma-full}, it suffices to verify that
    \begin{equation}\label{eqn:lemma-goal-A}
        \sum_{S\in C_2}g\left(\frac{b(s)}{\sqrt{|S|}}\right) = \sum_{i = 0}^{\lfloor\log_2 m\rfloor}f_i\cdot g\left(b(2^i)2^{-i/2}\right) \leq m/5.
    \end{equation}
	For fixed $i$, we denote $\tau = 2^i/M$, so $2^i = \tau M \geq \tau\phi(n)$. Observe that 
	\[f_i \leq c_0m \cdot \tau^{-1} \cdot \frac{\log \frac{en}{2^i}}{\log\frac{en}{\phi(n)}} \leq c_0m \tau^{-1}\frac{\log \frac{en}{\tau\phi(n)}}{\log\frac{en}{\phi(n)}} = c_0m\tau^{-1}\left(1+\frac{\log \tau^{-1}}{\log\frac{en}{\phi(n)}}\right).\]
	When $\tau < 1$, we have $b(2^i)2^{-i/2} = c_1\tau^{-0.1}$ and
	\[f_i \leq  c_0m\tau^{-1}\left(1+\frac{\log \tau^{-1}}{\log\frac{en}{\phi(n)}}\right) \leq c_0m\tau^{-1}(1+\log \tau^{-1}).
	\]
	When $\tau \geq 1$, we have $b(2^i)2^{-i/2} = c_1\tau^{-1}$ and 
	\[f_i \leq c_0m\tau^{-1}\left(1+\frac{\log \tau^{-1}}{\log\frac{en}{\phi(n)}}\right) \leq c_0m\tau^{-1}.\]
	Therefore if we write the summation in \eqref{eqn:lemma-goal-A} in terms of $\tau$, we have
	\begin{equation}\label{eqn:goal-as-tau-A}
	    \sum_{i = 0}^{\lfloor\log_2 m\rfloor}f_i\cdot g\left(b(2^i)2^{-i/2}\right) = c_0m\left(\sum_{\tau \geq 1}\tau^{-1}g(c_1\tau^{-1}) + \sum_{\tau < 1}\tau^{-1}(1+\log \tau^{-1})g(c_1\tau^{-0.1})\right),
	\end{equation}
	where the summation of $\tau$ is over a geometric sequence with ratio $2$. By definition of $g$ in \eqref{eqn:def-g}, $g$ is monotonically decreasing. Let $T$ be a large absolute constant to be determined later. We have
	\[\sum_{\tau \geq T}\tau^{-1}g(c_1\tau^{-1}) \leq \sum_{\tau \geq T}\tau^{-1}g(2\tau^{-1}) = \sum_{\tau \geq T}\tau^{-1}\cdot 10\log(1+\tau).\]
	Since $\sum_{\tau \geq 1}\tau^{-1}\cdot 10\log(1+\tau)$ converges, there exists sufficiently large constant $T$ satisfying
	\begin{equation}\label{eqn:term-1-A}
	    \sum_{\tau \geq T}\tau^{-1}g(c_1\tau^{-1}) \leq \frac{1}{20c_0}.
	\end{equation}
	We bound the second term on the right hand side of \eqref{eqn:goal-as-tau-A} similarly. We have
	\[\sum_{\tau < T^{-1}}\tau^{-1}(1+\log \tau^{-1})g(c_1\tau^{-0.1}) \leq \sum_{\tau < T^{-1}} \tau^{-1}(1+\log \tau^{-1})g(2\tau^{-0.1}) = \sum_{\tau < T^{-1}} \tau^{-1}(1+\log \tau^{-1})\cdot 10e^{-\tau^{-0.2}}.\]
	Since $\sum_{\tau < 1} \tau^{-1}(1+\log \tau^{-1})\cdot 10e^{-\tau^{-0.2}}$ converges, there exists sufficiently large constant $T$ satisfying
	\begin{equation}\label{eqn:term-2-A}
	    \sum_{\tau < T^{-1}}\tau^{-1}(1+\log \tau^{-1})g(c_1\tau^{-0.1}) \leq \frac{1}{20c_0}.
	\end{equation}
	Hence, there exists constant $T$ such that whenever $c_1 > 2$, \eqref{eqn:term-1-A} and \eqref{eqn:term-2-A} hold. Note that there are at most $(1+\log_2T)$ terms in each of the ranges $1\leq \tau < T$ and $T^{-1} \leq \tau < 1$, and that $g(x)$ is monotonically decreasing and tends to zero as $x$ goes to infinity. We can choose $c_1 > 2$ sufficiently large so that
	\begin{equation}\label{eqn:term-3-A}
	    \sum_{1\leq \tau < T}\tau^{-1}g(c_1\tau^{-1}) \leq (1+\log_2T)\cdot g(c_1T^{-1}) \leq \frac{1}{20c_0}
	\end{equation}
	and 
	\begin{equation}\label{eqn:term-4-A}
	    \sum_{T^{-1} \leq  \tau < 1}\tau^{-1}(1+\log \tau^{-1})g(c_1\tau^{-0.1}) \leq (1+\log_2T)\cdot T(1+\log T)g(c_1) \leq \frac{1}{20c_0}.
	\end{equation}
	
	Combining inequalities \eqref{eqn:term-1-A}, \eqref{eqn:term-2-A}, \eqref{eqn:term-3-A} and \eqref{eqn:term-4-A}, there exists a constant $c_1$ such that
	\begin{equation}\label{eqn:bound-on-sums}
	\sum_{\tau \geq 1}\tau^{-1}(1+\log\tau)g(c_1\tau^{-1}) + \sum_{\tau < 1}\tau^{-1}g(c_1\tau^{-0.1}) \leq \frac{1}{5c_0}.\end{equation}
	We use this constant $c_1$ to define the function $b(\cdot)$ in \eqref{eqn:b-A}. Combining \eqref{eqn:goal-as-tau-A} and \eqref{eqn:bound-on-sums}, we have
	\[\sum_{i = 0}^{\lfloor\log_2 m\rfloor}f_i\cdot g\left(b(2^i)2^{-i/2}\right) \leq c_0m \cdot \frac{1}{5c_0} = \frac{m}{5}.\]
	Therefore \eqref{eqn:lemma-goal-A} holds. By Lemma~\ref{lem:partial-coloring-lemma-full}, there exists a partial coloring $\chi: X\to \{-1, 0, 1\}$ that assigns $\pm 1$ to at least $m/10$ elements in $X$ such that \eqref{eqn:lem-outcome-A} holds. For this coloring $\chi$, we may choose $\Delta_i = b(2^i)$ in \eqref{eqn:decomp-A}. Again we denote $\tau = 2^i/M$, or equivalently $2^i = \tau M$. We have
	\[\Delta_i = b(2^i) = \begin{cases} c_1\sqrt{M}\cdot \tau^{-0.5} & \textrm{if } \tau \geq 1, \\ c_1\sqrt{M}\cdot \tau^{0.4} & \textrm{if } \tau < 1. \end{cases}\]
	Put this in \eqref{eqn:decomp-A}. We know that $\chi$ satisfies that
	\[\max_{A\in \CA_n}|\chi(A\cap X)| \leq 4\sum_{0\leq i\leq \log_2m}\Delta_i = 4c_1\sqrt{M}\left(\sum_{\tau \geq 1}\tau^{-0.5} + \sum_{\tau < 1}\tau^{0.4}\right).\]
	Note that summation of $\tau$ is over a geometric sequence with ratio $2$. Hence we have 
	\[\sum_{\tau \geq 1}\tau^{-0.5} \leq \frac{1}{1-2^{-0.5}} \quad \AND \quad \sum_{\tau < 1}\tau^{0.4}\leq \frac{1}{1-2^{-0.4}}.\]
	Therefore we conclude that we can find $\chi$ that assigns $\pm1$ to at least $m/10$ elements in $X$ and satisfies
	\[ \max_{A\in \CA_n}|\chi(A\cap X)| \leq 4c_1\sqrt{M}\left(\frac{1}{1-2^{-0.5}} + \frac{1}{1-2^{-0.4}}\right) \leq c\sqrt{M} = c\phi(n)^{\frac12}\left(\log\frac{en}{\phi(n)}\right)^{\frac12}\]
	for some appropriate absolute constant $c$.\end{proof}
	
\end{document}